\documentclass{amsart}
\usepackage[colorlinks,citecolor=blue,urlcolor=blue,linkcolor=blue]{hyperref}
\usepackage{amsmath}
\usepackage{amsfonts}
\usepackage{amssymb}
\usepackage[left=2.9cm,right=2.9cm,top=3cm,bottom=3cm]{geometry}

\def\R{{\mathbb R}}

\def\N{{\mathbb N}}

\def\<{\langle}
\def\>{\rangle}
\def\P{\mathbb P}

\def\E{\mathbb E}

\def\0{\underline 0}

\def\1{\underline 1}

\numberwithin{equation}{section}

\newcommand{\bel}{\begin{equation}\label}
	\newcommand{\nobel}{\begin{equation}}
		\newcommand{\ee}{\end{equation}}
	\newtheorem{theorem}{Theorem}[section]
	
	\newtheorem{corollary}[theorem]{Corollary}
	\newtheorem{lemma}[theorem]{Lemma}
	\newtheorem{remark}{Remark}[section]
	\theoremstyle{definition}

    \title[Probability laws associated to the independence preserving quadrirational Yang-Baxter  maps]{Probability laws associated to \\ the independence preserving quadrirational Yang-Baxter  maps  - the ultimate case}
	\author{Bartosz Ko{\l}odziejek}
        \address{Faculty of Mathematics and Information Sciences, Warsaw University of Technology, Koszykowa 75, \mbox{00-662} Warsaw, Poland}
\email{bartosz.kolodziejek@pw.edu.pl}
	\author{G\'erard Letac}
    \address{Laboratoire de Statistique et Probabilit\'es, Universit\'e Paul Sabatier, 118 Route de Narbonne, 31062 Toulouse,  \& TESA, 7 Bd de la Gare, 31500, Toulouse, France}
    \email{gerard.letac@math.univ-toulouse.fr}
	    \author{Mauro Piccioni}
    \address{Dipartimento di Matematica, Sapienza Universit\`a di Roma, Piazzale Aldo Moro 5, 00185 Roma, Italy.}
    \email{mauro.piccioni@uniroma1.it}
    
	\author{Jacek Weso{\l}owski}
 \address{Faculty of Mathematics and Information Sciences, Warsaw University of Technology, Koszykowa 75, \mbox{00-662} Warsaw,
Poland}
\email{jacek.wesolowski@pw.edu.pl}

\begin{document}
		\begin{abstract}
A map $F\colon\mathcal X\times\mathcal Y\to \mathcal U\times \mathcal V$ is said to be independence preserving (IP) if there exists a pair of independent random variables $(X,Y)$ valued in $\mathcal X\times\mathcal Y$ such that the two coordinates of  $(U,V)=F(X,Y)\in\mathcal U\times\mathcal V$ are also independent. Recently, Sasada and Uozumi (2024) observed that a hierarchy of quadrirational Yang-Baxter maps gives rise to independence preserving transformations, and identified corresponding families of probability distributions. 
In view of the limiting properties of these IP maps, the newly defined generalized second kind beta ($\mathrm{GB}_{II}$) model stands at the top of the hierarchy: for independent random variables $X$ and $Y$ following a $\mathrm{GB}_{II}$ distribution, Sasada and Uozumi (2024) showed that when a special quadrirational Yang-Baxter map $F^{(\alpha,\beta)}$, parameterized by  $(\alpha,\beta)\in(0,\infty)^2$, is applied to the pair $(X,Y)$, it produces another pair $(U,V)$ of independent  $\mathrm{GB}_{II}$-distributed random variables. Interestingly, the boundary cases of $\alpha\in\{0,\infty\}$ (or $\beta\in\{0,\infty\}$) are related to one of the Matsumoto-Yor IP maps identified in Koudou and Vallois (2012). 
			
The aim of this paper is to show that the IP property of $F^{(\alpha,\beta)}$ uniquely identifies distributions of $X,Y,U$ and $V$ as belonging to the $\mathrm{GB}_{II}$ family. To this end, we introduce specially designed Laplace-type transforms. First, we carefully explain the connection between the results from  Sasada and Uozumi (2024) and Koudou and Vallois (2012). Next, we focus on the characterization of the second kind beta and the generalized second kind beta distributions through the IP map $F^{(\alpha,\infty)}$. Finally, extending considerably the methodology developed for the case $(\alpha,\infty)$, we prove the characterization of $\mathrm{GB}_{II}$ distributions in the case $(\alpha,\beta)\in(0,\infty)^2$ with $\alpha\neq\beta$, which implies uniqueness in the ultimate missing case of the quadrirational Yang-Baxter hierarchy of IP models. 
		\end{abstract}
		
		\maketitle
		
				\section{Preliminaries}
		Let $(X,Y)$ be a pair of independent random variables valued in $\mathcal X\times\mathcal Y$. We say that a function $F\colon\mathcal X\times \mathcal Y\to\mathcal U\times\mathcal V$ is an IP (independence preserving) map if there exists a pair $(X,Y)$ of independent  $\mathcal X\times\mathcal Y$-valued random variables such that $(U,V)=F(X,Y)$ is a pair of independent $\mathcal U\times\mathcal V$-valued random variables. For several such IP maps, corresponding  families of distributions of $X$, $Y$, $U$ and $V$ have been identified.  Some of these results are quite classical: e.g., the Kac-Bernstein characterization of the normal law for the map $F(x,y)=(x+y,x-y)$ (defined on $\mathbb R^2$) \cite{Kac,Ber}, or the Lukacs characterization of the gamma laws for the map $F(x,y)=(x+y,x/y)$ (defined on $(0,\infty)^2$) \cite{Lukacs}. A recent revival of interest in this area is due to the fact that IP maps are fundamental for the construction of integrable probabilistic models  (see, e.g., \cite{CroSas2020}, \cite{CroSas2022}, \cite{AriBisOCo2023}), in particular the Lukacs property is crucial for the log-gamma polymer (see, e.g., \cite{Sep2012}). Indeed, all four known types of $1:1$ random directed polymers are based on IP maps related to independence characterization of gamma or beta distributions, see \cite{CorSepShe2015}, \cite{OCoOrt2015}, \cite{BarCor2017}, \cite{ThiLeD2017}, \cite{ChaNoa2018}.    
		
		Recently, in \cite{SasUoz2024} the authors discussed a hierarchy of so-called $[2:2]$ quadrirational Yang-Baxter transformations, all of which are IP maps.  This family was introduced in \cite{PSTV2010} (see also \cite{ABS2004}). These maps include IP maps for the generalized inverse Gaussian laws (they generalize the Matsumoto-Yor property discovered in \cite{MatYor2001}, \cite{MatYor2003}), see \cite{LetWes2024} for a related characterization, as well as the IP maps for the Kummer distribution, see \cite{HamzaVallois}, \cite{KouWes2024}. At the top of the hierarchy described in \cite{SasUoz2024} is the function $F^{(\alpha,\beta)}\colon(0,\infty)^2\to(0,\infty)^2$, $\alpha,\beta>0$ defined by
		\begin{equation}\label{Fab}
			F^{(\alpha,\beta)}(x,y)=\left(\tfrac{y}{\alpha}\,\tfrac{\beta+\alpha x+\beta y+\alpha\beta xy}{1+x+y+\beta xy},\;\tfrac{x}{\beta}\,\tfrac{\alpha+\alpha x+\beta y+\alpha\beta xy}{1+x+y+\alpha xy}\right);
		\end{equation}
		in particular, $F^{(\alpha,\alpha)}(x,y)=(y,x)$. It coincides with the map $H_I^+$ from \cite{PSTV2010}.  Many well-known examples of IP mappings can be derived from $F^{(\alpha,\beta)}$ by taking special parameters or
		performing a singular limit with an appropriate coordinate-wise change of variables, see the diagram at page 4 and Section 3.2 in \cite{SasUoz2024}. 
		
		It was shown in \cite{SasUoz2024} that $F^{(\alpha,\beta)}$ is an IP map for a distribution that the authors call the generalized second kind beta.  This distribution, that we denote by $\mathrm{GB}_{II}(\nu,p,q;\gamma)$, is defined through its density
		\begin{equation}\label{dens}
			f(x) =\tfrac{1}{B(q+\nu,p-\nu)\,{}_2F_1(p+\nu,q+\nu;p+q;1-\gamma)}\,\tfrac{x^{q+\nu-1}}{(1+\gamma x)^{p+\nu}(1+x)^{q-\nu} }\,\mathbf 1_{(0,\infty)}(x),
		\end{equation}
		with $\gamma\ge 0$, $p, q>0$ and $-q<\nu<p$. Here $B$ is the beta function,
		$$B(a,b)=\tfrac{\Gamma(a)\Gamma(b)}{\Gamma(a+b)},\quad a,b>0,$$ 
		and $_2F_1$ is the Gauss hypergeometric function, defined as follows:
		$$
		_2F_1(a,b,c;z)=\tfrac1{B(b,c-b)}\,\int_0^1\,t^{b-1}(1-t)^{c-b-1}(1-zt)^{-a}\,dt, \quad c>b>0,\;z<1.
		$$
		Note that $\mathrm{GB}_{II}(\nu,p,q;1) = \mathrm{B}_{II}(\nu+q, p-\nu)$ and,  for $\nu <0$,  $\mathrm{GB}_{II}(\nu,p,q;0) = \mathrm{B}_{II}(\nu+q, -2\nu)$, where $\mathrm{B}_{II}(a,b)$, $a,b>0$ stands for the standard beta distribution of the second kind,  defined by the density
		$$
		f(x)=\tfrac1{B(a,b)} \tfrac{x^{a-1}}{(1+x)^{a+b}}\mathbf 1_{(0,\infty)}(x).
		$$
		
		The following result is given in \cite{SasUoz2024}:
		\begin{theorem}\label{SU22}
			Let $\alpha,\beta>0$, $\lambda\in\R$, $a,b>0$ be such that $|\lambda|<\min\{a,b\}$. Assume that random variables $X$ and $Y$ satisfy
			\begin{equation}\label{XY}
				(X,Y)\sim \mathrm{GB}_{II}(\lambda,a,b;\alpha)\otimes \mathrm{GB}_{II}(-\lambda,a,b;\beta).
			\end{equation}
			Then, $(U,V)=F^{(\alpha,\beta)}(X,Y)$ satisfies
			\begin{equation}\label{UV}
				(U,V)\sim\mathrm{GB}_{II}(-\lambda,a,b;\alpha)\otimes \mathrm{GB}_{II}(\lambda,a,b;\beta).
			\end{equation}
		\end{theorem}	
		
		The main goal of this paper is to prove the converse result, providing a characterization of $\mathrm{GB}_{II}$  distributions in terms of the IP map $F^{(\alpha,\beta)}$:
		\begin{theorem}\label{charac}
			Assume that $X$ and $Y$ are non-Dirac, positive random variables that are independent. For $\alpha,\beta>0$ with $\alpha\neq \beta$, set $(U,V)=F^{(\alpha,\beta)}(X,Y)$. 
			
			If $U$ and $V$ are independent, then there exist $a,b>0$ and $\lambda\in\mathbb R$ satisfying $\min\{a,b\}>|\lambda|$ such that \eqref{XY} (and consequently \eqref{UV}) hold.
		\end{theorem}	
		This result covers the ultimate missing case of characterizations of IP models from the quadrirational Yang-Baxter hierarchy of Sasada and Uozumi, \cite{SasUoz2024}, with previous cases covered in \cite{LetWes2024} and \cite{KouWes2024}. 
        
        Theorem \ref{charac}  will be proved in Section \ref{proof}, which contains the main technical part of the paper. Earlier, in Section \ref{LTR}, we introduce a hypergeometric-type version of the Laplace transform and show how useful it is for analyzing the independence property given in Theorem \ref{SU22}. In Section \ref{KUVA}, we observe that the Matsumoto-Yor-type independence property of \cite{KouVal2012} (for the first kind beta and the generalized first kind beta law) is equivalent to the independence property generated by the map $F^{(\alpha,\beta)}$ with boundary values $\alpha=\infty$ or $\beta=\infty$. The main result of this section is a characterization of the $\mathrm B_{II}$ and $\mathrm{GB}_{II}$ distributions in this boundary case, its proof being a warm-up for the arguments of Section \ref{proof}, where we considerably develop the methodology of Section \ref{KUVA}.   
		
		\section{Hypergeometric-type Laplace transforms}\label{LTR}

        For $\gamma>0$, $s+\theta, \sigma \ge 0$ and a non-negative random variable $W$  denote
		\begin{equation}\label{Lw}
			L_W^{(\gamma)}(\theta,\sigma,s)=\gamma^\theta\,\E\left[\tfrac{ W^{\theta+s}}{(1+\gamma W)^{\theta}(1+W)^{\sigma+s}}\right]=\E\left[(\tfrac{\gamma W}{1+\gamma W})^{\theta} (\tfrac {1}{1+W})^{\sigma} (\tfrac{W}{1+W})^s\right].
		\end{equation}
		It is clear that $L_W^{(\gamma)}$ uniquely determines the law of $W$. Indeed it is  determined, e.g., by $L_W^{(\gamma)}(k,0,0)$ for $k=1,2,\ldots$. 
		The trivial identities
		$$
		1=\frac {w}{1+w}+\frac {1}{1+w}=\frac {\gamma w}{1+\gamma w}+\frac {1}{\gamma w}\,\frac {\gamma w}{1+\gamma w}
		$$
		imply
		\begin{equation}\label{id1}
			L_W^{(\gamma)}(\theta,\sigma,s)=L_W^{(\gamma)}(\theta+1,\sigma,s)+\gamma^{-1}L_W^{(\gamma)}(\theta+1,\sigma+1,s-1),
		\end{equation}
		and
		\begin{equation}\label{id2}
			L_W^{(\gamma)}(\theta,\sigma,s)=L_W^{(\gamma)}(\theta,\sigma,s+1)+L_W^{(\gamma)}(\theta,\sigma+1,s)
		\end{equation}
		for $\gamma>0$ and $\sigma, \theta+s \ge 0$.\footnote{Identities \eqref{id1} and \eqref{id2} can be understood as linearization properties of difference operators $\Delta_1$ and $\Delta_2$ applied to $L_W$, see \eqref{W}, and together with \eqref{D1D2L} they yield a bilinearization formula for $L_W\,\Delta_1\Delta_2\,L_W-\Delta_2L_W\,\Delta_1\,L_W$, which suggests a link to classical integrable systems.  We thank M. Sasada and R. Willox for discussions on such a connection, however it is not explored further in this paper.	} 
		
		In view of \eqref{dens}, for $W\sim\mathrm{GB}_{II}(\nu,p,q;\gamma)$ we have 
		\begin{equation}\label{LwB}
			L_W^{(\gamma)}(\theta,\sigma,s)=\gamma^\theta\,\tfrac{B(s+\theta+q+\nu,\sigma+p-\nu)_2F_1(\theta+p+\nu,s+\theta+q+\nu;s+\theta+\sigma+p+q;1-\gamma)}{B(q+\nu,p-\nu)\,_2F_1(p+\nu,q+\nu;p+q;1-\gamma)}
		\end{equation}
		\begin{equation}\label{LwB2}
			=\gamma^{\theta}\frac {(q+\nu)^{(s+\theta)}(p-\nu)^{(\sigma)}}{(p+q)^{(s+\theta+\sigma)}}
			\tfrac{_2F_1(\theta+p+\nu,s+\theta+q+\nu;s+\theta+\sigma+p+q;1-\gamma)}{_2F_1(p+\nu,q+\nu;p+q;1-\gamma)}
		\end{equation}
		where $c^{(d)}=\Gamma(c+d)/\Gamma(c)$; in particular, for $d\in \mathbb N$ , $c^{(d)}=\prod_{j=0}^{d-1}\,(c+j)$ is the ascending Pochhammer symbol. 
		
		Throughout the paper, we denote $\mathbb{N}=\{1,2,\ldots\}$ and $\mathbb{N}_0 = \{0,1,\ldots\}$. 
		
		The original proof of Theorem \ref{SU22}, as presented in \cite{SasUoz2024}, involves computing the Jacobian of $F^{(\alpha,\beta)}$ and showing the corresponding identity for densities. Here, we provide an alternative proof based on the transforms \eqref{Lw} of $X,Y,U,V$.

		\begin{proof}[Proof of Theorem \ref{SU22}]
			By straightforward verification, we check that for $x,y,u,v>0$ the relation 
			$$(u,v)=F^{(\alpha,\beta)}(x,y)$$ holds  if and only if  (actually any two of the equalities below imply the third)
			\begin{align}\label{xyuv}\begin{split}
					\tfrac{xy}{(1+x)(1+y)}=\tfrac{uv}{(1+u)(1+v)}, \\
					\tfrac{\alpha x}{(1+\alpha x)(1+y)}=\tfrac{\beta v}{(1+u)(1+\beta v)}, \\
					\tfrac{\beta y}{(1+x)(1+\beta y)}=\tfrac{\alpha u}{(1+\alpha u)(1+v)}.
				\end{split}
			\end{align}
			
			Consider $(U,V)=F^{(\alpha,\beta)}(X,Y)$. Then \eqref{xyuv} and independence of $X$ and $Y$ yield
            \begin{align}\label{XYEE}
            &L_X^{(\alpha)}(\theta,\sigma,s)\,L_Y^{(\beta)}(\sigma,\theta,s)\\
            =&\E\left[\left(\tfrac{\beta V}{(1+U)(1+\beta V)}\right)^{\theta}
			\left(\tfrac{\alpha U}{(1+\alpha U)(1+V)}\right)^\sigma\left(\tfrac{UV}{(1+U)(1+V)}\right)^s\right],\quad (\theta,\sigma,s)\in\Xi,\nonumber
            \end{align}
			where 
            \begin{equation}\label{setF}
				\Xi=\{(\theta,\sigma,s)\in\R^3:\;\theta,\sigma,\theta+s,\sigma+s\ge 0\}.
			\end{equation} Note that the function on the right-hand side of \eqref{XYEE} uniquely determines the law of 
			\[
			\left(\tfrac{\beta V}{(1+U)(1+\beta V)},\,\tfrac{\alpha U}{(1+\alpha U)(1+V)},\,\tfrac{UV}{(1+U)(1+V)}\right)
			\]
            and thus the law of $(U,V)$. On the other hand for $U$ and $V$ independent the right-hand side of \eqref{XYEE} equals $L_U^{(\alpha)}(\sigma,\theta,s)\,L_V^{(\beta)}(\theta,\sigma,s)$. 
            
            Therefore, in view of equality \eqref{XYEE}, and uniqueness of the hypergeometric-type Laplace transforms $L_U^{(\alpha)}$ and $L_V^{(\beta)}$, to prove independence of $U$ and $V$ it suffices to show that for distributions of $X,Y,U,V$ specified in \eqref{XY} and \eqref{UV},
			\begin{equation}\label{Lindep}
				L_X^{(\alpha)}(\theta,\sigma,s)\,L_Y^{(\beta)}(\sigma,\theta,s)=	L_U^{(\alpha)}(\sigma,\theta,s)\,L_V^{(\beta)}(\theta,\sigma,s),\quad (\theta,\sigma,s)\in\Xi.
			\end{equation}	  
			
			Looking at the parameters of the distributions in \eqref{XY} and \eqref{UV}, we see that, to prove \eqref{Lindep}, it suffices to show that for any parameter $\gamma>0$, $X\sim \mathrm{GB}_{II}(\lambda,a,b;\gamma)$ and $U\sim \mathrm{GB}_{II}(-\lambda,a,b;\gamma)$, the quotient $\tfrac{L_X^{(\gamma)}(\theta,\sigma,s)}{L_U^{(\gamma)}(\sigma,\theta,s)}$ does not depend on the value of $\gamma>0$. 
			
			To this end, we refer to \eqref{LwB2}, which gives (recall that $(c)^{(b)}$ is the ascending Pochhammer symbol)
			\begin{multline}\label{nicer}
				\tfrac{L_X^{(\gamma)}(\theta,\sigma,s)}{L_U^{(\gamma)}(\sigma,\theta,s)}=
				\tfrac{(a-\lambda)^{(\sigma)}\,(b+\lambda)^{(s+\theta)}}{(a+\lambda)^{(\theta)}\,(b-\lambda)^{(s+\sigma)}}
				\\
				\times\gamma^{\theta-\sigma}\tfrac {_2F_1(\theta+a+\lambda,s+\theta+b+\lambda;s+\theta+\sigma+a+b;1-\gamma)_2F_1(a-\lambda,b-\lambda;a+b;1-\gamma)}{_2F_1(a+\lambda,b+\lambda;a+b;1-\gamma)_2F_1(\sigma+a-\lambda,s+\sigma+b-\lambda;s+\sigma+\theta+a+b;1-\gamma)}.
			\end{multline}
			By applying the Euler identity
			\begin{equation}\label{EI}
				_2F_1(A,B,C;Z)=(1-Z)^{C-A-B}\,_2F_1(C-A,C-B,C;Z)
			\end{equation}	
			with $$(A,B,C,Z)=(\theta+a+\lambda,\,s+\theta+b+\lambda,\,s+\theta+\sigma+a+b,1-\gamma)$$ and with $$(A,B,C,Z)=(a+\lambda,\,b+\lambda,\,a+b,1-\gamma),$$ and by using  obvious identity $_2F_1(A,B;C;Z)=\,_2F_1(B,A;C;Z)$, we find that the second factor in \eqref{nicer} is 1. Thus,
			\begin{equation}\label{LX/LU}
				\tfrac{L_X^{(\gamma)}(\theta,\sigma,s)}{L_U^{(\gamma)}(\sigma,\theta,s)}=\tfrac{(a-\lambda)^{(\sigma)}\,(b+\lambda)^{(\theta+s)}}{(a+\lambda)^{(\theta)}\,(b-\lambda)^{(\sigma+s)}},
			\end{equation}
			which does not depend on $\gamma$, thereby concluding the proof.
		\end{proof}
		
		\section{\texorpdfstring{The boundary case $\beta=\infty$}{The boundary case of beta=infty}}\label{KUVA}
		
		In this section, we intend to consider the boundary case of Theorems \ref{SU22} and \ref{charac} as $\beta\to \infty$. 
		Instead of proving a single  `if and only if ' result, we prefer to split it in two parts, Theorems  \ref{infin} and \ref{charac1} as we did in the introduction.   Although these statements can be seen as limits of Theorems \ref{SU22} and Theorem \ref{charac}, we provide a direct proof as a warm-up for the methods used in the longer proof in Section \ref{proof}. The remainder of the section recalls several results from the literature that are  equivalent to Theorem \ref{infin}  or are particular cases of Theorem \ref{charac1}. Finally, we will explain that the other boundary case $\beta=0$ can be reduced to the case $\beta=\infty$.

        By taking the limit $\beta\to\infty$ in \eqref{Fab}, we obtain
		\begin{equation}\label{Finf}
			F^{(\alpha,\infty)}(x,y)=\left(\tfrac{1+ y+\alpha xy}{\alpha  x},\;\tfrac{xy(1+\alpha x)}{1+x+y+\alpha xy}\right),
		\end{equation}
		which is an involution on $(0,\infty)^2$.  Similarly for $F^{(\infty,\beta)}$. 
		
		To study this boundary case, we allow $\gamma=\infty$ in $L_W^{(\gamma)}$, with the following definition
		$$
		L_W^{(\infty)}(\theta,\sigma,s)=\E\left[\tfrac{ W^s}{(1+W)^{\sigma+s}}\right]=:L_W^{(\infty)}(\sigma,s),
		$$
        for which
        \begin{equation}\label{added}
        L_W^{(\infty)}(\sigma,s)=L_W^{(\infty)}(\sigma,s+1)+L_W^{(\infty)}(\sigma+1,s)
        \end{equation}
		holds for $s,\sigma \geq 0$.
		
		Note that for $W\sim \mathrm{B}_{II}(a,b)$ we have
		\begin{equation}\label{lwb}
			L_W^{(\infty)}(\sigma,s)=\tfrac{B(a+s,b+\sigma)}{B(a,b)}, \, \sigma,\, s\geq 0.
		\end{equation}
		
		Here is the analogue of the  independence property from Theorem \ref{SU22} in the boundary case $\beta=\infty$.

		\begin{theorem}\label{infin}
			Let $|\lambda|<a<b$, $\alpha>0$. Assume that random variables $X$ and $Y$ satisfy
			\begin{equation}\label{XY1}
				(X,Y)\sim\mathrm{GB}_{II}(\lambda,a,b;\alpha)\otimes \mathrm{B}_{II}(b-a,a+\lambda).
			\end{equation}
			Then, $(U,V)=F^{(\alpha,\infty)}(X,Y)$ satisfies
			\begin{equation}\label{UV1}
				(U,V)\sim\mathrm{GB}_{II}(-\lambda,a,b;\alpha)\otimes \mathrm{B}_{II}(b-a,a-\lambda).
			\end{equation}
		\end{theorem}  
	 Under the assumptions \(|\lambda|<a<b\), the statement of Theorem \ref{infin}, as announced in the sentence preceding its formulation, coincides with the formal boundary limit \(\beta\to+\infty\) of Theorem \ref{SU22}. We derive it directly from Theorem \ref{AKPV} and Remark \ref{connect} below, which establish its equivalence to the independence property of \cite{KouVal2012}.
    The main result of this section, Theorem \ref{charac1} below, gives its converse.
		\begin{theorem}\label{charac1}
			Assume that $X$ and $Y$ are non-Dirac, positive random variables, which are independent. For $\alpha>0$ set $(U,V)=F^{(\alpha,\infty)}(X,Y)$.
			
			If $U$ and $V$ are independent, then there exist constants $a,b,\lambda$ with $b>a>|\lambda|$ such that \eqref{XY1} (and consequently \eqref{UV1}) hold.
		\end{theorem}
		
		\begin{proof}[Proof of Theorem \ref{charac1}] 
			The independencies of $X$ and $Y$, and of $U$ and $V$, yield \begin{equation}\label{Lindep1}
				L_X^{(\alpha)}(\theta,\sigma,s)\,L_Y^{(\infty)}(\theta,s)=	L_U^{(\alpha)}(\sigma,\theta,s)\,L_V^{(\infty)}(\sigma,s).
			\end{equation}	
            Expand $L_Y^{(\infty)}(\theta,s)$ according to \eqref{added} and plug it in \eqref{Lindep1}. Next, expand only the factor multiplying $L_Y^{(\infty)}(\theta+1,s)$ according to \eqref{id1}. By equating it with the right-hand side of \eqref{Lindep1}, with $L_U^{(\alpha)}(\sigma,\theta,s)$ expanded according to \eqref{id2}, and deleting the term appearing at both sides, we see that 
			\begin{align*}
				&\tfrac1{\alpha}L_X^{(\alpha)}(\theta+1,\sigma+1,s-1)L_Y^{(\infty)}(\theta+1,s)+L_X^{(\alpha)}(\theta,\sigma,s)L_Y^{(\infty)}(\theta,s+1)\\
				&\qquad\qquad=L_U^{(\alpha)}(\sigma,\theta,s+1)L_V^{(\infty)}(\sigma,s),
			\end{align*}
			holding for $\theta,\sigma\geq 0$ and $s\ge 1$. In view of the fact that $F^{(\alpha,\infty)}$ is an involution, by symmetry
			\begin{align*}&\tfrac 1\alpha\,L_U^{(\alpha)}(\sigma+1,\theta+1,s-1)L_V^{(\infty)}(\sigma+1,s)+L_U^{(\alpha)}(\sigma,\theta,s)L_V^{(\infty)}(\sigma,s+1)\\
				&\qquad\qquad =L_X^{(\alpha)}(\theta,\sigma,s+1)L_Y^{(\infty)}(\theta,s).
			\end{align*}
			Cross-multiplying the above two equalities, and exploiting \eqref{Lindep1} both for $s$ and $s+1$, we obtain 
			\begin{multline*}
				L_X^{(\alpha)}(\theta,\sigma,s+1)L_Y^{(\infty)}(\theta,s)L_X^{(\alpha)}(\theta+1,\sigma+1,s-1)L_Y^{(\infty)}(\theta+1,s)\\=L_U^{(\alpha)}(\sigma,\theta,s+1)L_V^{(\infty)}(\sigma,s)L_U^{(\alpha)}(\sigma+1,\theta+1,s-1)L_V^{(\infty)}(\sigma+1,s).
			\end{multline*}
			Referring again to \eqref{Lindep1}, we arrive at
			\begin{equation}\label{sepvar}
				\tfrac{L_Y^{(\infty)}(\theta,s)\,L_Y^{(\infty)}(\theta+1,s)}{L_Y^{(\infty)}(\theta,s+1)\,L_Y^{(\infty)}(\theta+1,s-1)}=\tfrac{L_V^{(\infty)}(\sigma,s)\,L_V^{(\infty)}(\sigma+1,s)}{L_V^{(\infty)}(\sigma,s+1)\,L_V^{(\infty)}(\sigma+1,s-1)}.
			\end{equation}
			Setting
			$$
			R(\theta,\sigma,s)=\tfrac{L_Y^{(\infty)}(\theta,s)}{L_Y^{(\infty)}(\theta+1,s-1)}\,\tfrac{L_V^{(\infty)}(\sigma+1,s-1)}{L_V^{(\infty)}(\sigma,s)},
			$$
			equality \eqref{sepvar} implies 
			$$
R(\theta,\sigma,s+1)=R(\theta,\sigma,s)=R(\theta,\sigma,1)=\tfrac{R_V(\sigma)}{R_Y(\theta)},\quad s \in \mathbb N,
			$$
			where $R_Y(\theta)=\tfrac{L_Y^{(\infty)}(\theta+1,0)}{L_Y^{(\infty)}(\theta,1)}$ and $R_V(\sigma)=\tfrac{L_V^{(\infty)}(\sigma+1,0)}{L_V^{(\infty)}(\sigma,1)}$.
			Consequently,
			\begin{equation}\label{Mauro1}
				\tfrac{L_Y^{(\infty)}(\theta,s)}{L_V^{(\infty)}(\sigma,s)}=\tfrac{L_Y^{(\infty)}(\theta+1,s-1)}{L_V^{(\infty)}(\sigma+1,s-1)}\,\tfrac{R_V(\sigma)}{R_Y(\theta)}=\tfrac{L_Y^{(\infty)}(\theta+s,0)}{L_V^{(\infty)}(\sigma+s,0)} \,\tfrac{R_V(\sigma)\cdots R_V(\sigma+s-1)}{R_Y(\theta)\cdots R_Y(\theta+s-1)}.
			\end{equation}
			Introduce now $N_1$, $N_2$, $N_3$, $N_4$ as follows
            \begin{gather*}
            N_3(x):=\prod_{j=0}^{x-1}\,R_Y(j),\quad N_2(x):=\prod_{j=0}^{x-1}\,R_V(j),\quad x \in \mathbb N, \quad N_3(0)=N_2(0)=1,
            \\
              N_4(x):= \frac {N_2(x)}{L_V^{(\infty)}(x,0)},\quad  N_1(x):=\frac {N_3(x)}{L_Y^{(\infty)}(x,0)}, \quad x\in\mathbb N_0. 
            \end{gather*}
        
            Then \eqref{Mauro1} can be written as
			\begin{equation}\label{N1234}
				L_V^{(\infty)}(\sigma,s)=\tfrac{N_1(\theta+s)N_2(\sigma)}{N_3(\theta)N_4(\sigma+s)}\,L_Y^{(\infty)}(\theta,s).
			\end{equation}
			Let $n_i(k):=\tfrac{N_i(k+1)}{N_i(k)}$ for $k=0,1,\ldots$ and $i=1,2,3,4$. Equating the right-hand side of \eqref{N1234} for $\theta$ and $\theta+1$, and upon referring to \eqref{added} for $L_Y^{(\infty)}$, we obtain 
			\begin{equation}\label{n13**}
				(n_1(\theta+s)-n_3(\theta))L_Y^{(\infty)}(\theta,s)=n_1(s+\theta)L_Y^{(\infty)}(\theta,s+1).
			\end{equation}
			On the other hand, applying \eqref{N1234} to both sides of $L_V^{(\infty)}(\sigma,s+1)=L_V^{(\infty)}(\sigma,s)-L_V^{(\infty)}(\sigma+1,s)$,  after cancellations of the common factors, we arrive at
			\begin{equation}\label{n42**}
				(n_4(\sigma+s)-n_2(\sigma))L_Y^{(\infty)}(\theta,s)=n_1(s+\theta)L_Y^{(\infty)}(\theta,s+1).
			\end{equation}
			Equating the right-hand sides of \eqref{n13**} and \eqref{n42**}, we obtain for $s,\theta,\sigma\in\mathbb N_0$
			\begin{equation}\label{CPE}
			n_1(\theta+s)-n_3(\theta)=n_4(\sigma+s)-n_2(\sigma)=:f(s),
			\end{equation}
			where the last equality is due to the separation of variables. We obtain two Cauchy-Pexider equations on $\mathbb{N}_0$ (see e.g. \cite{Acz1966}), $n_3(\theta)+f(s)=n_1(\theta+s)$ and $n_2(\sigma)+f(s)=n_4(\sigma+s)$. We conclude that $n_i$, $i=1,2,3,4$ and $f$ are affine functions with the same slope, say $q$. Notably, $q=0$ is impossible; in such case, substituting  $s=\theta=0$ in \eqref{N1234} would yield $L_V^{(\infty)}(\sigma,0)=a^\sigma$, implying $V\sim \delta_{\frac{1-a}{a}}$, which leads to a contradiction. Thus, there exist $q\neq 0$ and $a_i\in \mathbb R$ such that $a_1+a_2=a_3+a_4$, and
			$$
			n_i(s)=q(a_i+s),\quad i=1,2,3,4.
			$$
			In view of \eqref{N1234}, for $s=\sigma=0$ and for $s=\theta=0$, we obtain 
			$$
			L_Y^{(\infty)}(\theta,0)=\tfrac{a_3^{(\theta)}}{a_1^{(\theta)}}\quad\mbox{and}\quad L_V^{(\infty)}(\sigma,0)=\tfrac{a_2^{(\sigma)}}{a_4^{(\sigma)}},\quad \theta,\sigma=0,1,\ldots.
			$$
			Since $L_Y^{(\infty)}(\theta,0)=\E[(1+Y)^{-\theta}]$ and $L_V^{(\infty)}(\sigma,0)=\E[(1+V)^{-\sigma}]$ are in $(0,1)$, it follows that $a_1>a_3=:a+\lambda>0$ and $a_4>a_2=:a-\lambda>0$. Consequently, we conclude that $Y\sim\mathrm{B}_{II}(b-a,a+\lambda)$ and $V\sim\mathrm{B}_{II}(b-a,a-\lambda)$, where $b-a:=a_1-a_3=a_4-a_2>0$.
			
			Clearly, $Y$ and $V$ are absolutely continuous. Since $U=F_1^{(\alpha,\infty)}(X,Y)$ is a smooth function of independent $(X,Y)$, and $Y$ has a density, $U$ also has a density, say $f_U$. Similarly, we conclude that $X$ has a density, say $f_X$. Therefore, the independence assumptions yield\footnote{Here and below, the notation $f\propto g$, where $f$ and $g$ are two functions, means that $f=c\,g$ for some positive constant $c$.}
			\begin{align*}
				f_U(u)\tfrac{v^{b-a-1}}{(1+v)^{b-\lambda}}\propto |J(u,v)|\,f_X\left(\tfrac{1+ v+\alpha uv}{\alpha  u}\right)
				\left(\tfrac{uv(1+\alpha u)}{1+u+v+\alpha uv}\right)^{b-a-1}\left(1+\tfrac{uv(1+\alpha u)}{1+u+v+\alpha uv}\right)^{-b-\lambda},
			\end{align*}
			where 
			\[
			|J(u,v)| = \tfrac{(1+\alpha u)(1+v+\alpha uv)}{\alpha u(1+u+v+\alpha uv)}
			\]
			is the Jacobian determinant of $(F^{(\alpha,\infty)})^{-1}$. 
			This can be rewritten as
			\begin{equation}\label{fu}
				f_U(u)\propto \tfrac{u^{b-a-2}\,(1+\alpha u)^{b-a}\,(1+u+v+\alpha uv)^{a+\lambda}\,(1+v)^{b-\lambda}}{(1+u)^{b+\lambda}\,(1+v+\alpha uv)^{b+\lambda-1}}\,f_X\left(\tfrac{1+ v+\alpha uv}{\alpha  u}\right).
			\end{equation}
			Fix $\kappa>0$. Then, for every $u>(\kappa\alpha)^{-1}$ there exists $v>0$ such that $\tfrac{1+ v+\alpha uv}{\alpha  u}=\kappa$. Moreover, $1+v\propto\tfrac{u}{1+\alpha u}$ and $1+u+v+\alpha uv\propto u$. For such $v$, after simple algebra,  equality \eqref{fu} gives
			$$
			f_U(u)\propto \tfrac{u^{b-\lambda-1}}{(1+u)^{b+\lambda}(1+\alpha u)^{a-\lambda}}, \quad u>(\kappa\alpha)^{-1}.
			$$
			Since $\kappa$ can be taken arbitrarily large, we conclude that $U\sim\mathrm{GB}_{II}(-\lambda,a,b;\alpha)$. Finally, $f_X$ can be computed from \eqref{fu} yielding  $X\sim\mathrm{GB}_{II}(\lambda,a,b;\alpha)$.
		\end{proof}
		
		\subsection{Discussion of known results} 
        
    Theorem \ref{AKPV} below is due to Koudou and Vallois,  \cite{KouVal2012}, and  Remark  \ref{connect} will prove that our Theorem \ref{infin}  is equivalent to it.  To state it, as in \cite{KouVal2012}, we denote by  
		$\mathrm{GB}_{I}(p,q,r;\delta)$, $p,q>0$, $r\in\mathbb R$, $\delta>0$, a generalized first kind beta distribution defined by the density
		$$
		f(x)\propto x^{p-1}(1-x)^{q-1}(1+(\delta-1)x)^r\,\mathbf 1_{(0,1)}(x),
		$$
		and by $\mathrm{B}_I(p,q)$  the standard first kind beta distribution, obtained by taking $\delta=1$ or $r=0$ in $\mathrm{GB}_{I}(p,q,r;\delta)$. The independence property reads:
		
		\begin{theorem}[\cite{KouVal2012}]\label{AKPV}   
			Let \begin{equation}\label{x'y'}
				(X',Y')\sim \mathrm{GB}_{I}(a+b,c,-b-c;\delta)\otimes \mathrm{B}_I(a,b)
			\end{equation}
			for $a,b,c>0$, and let $G^{(\delta)}\colon(0,1)^2\to(0,1)^2$ be defined by
			\begin{equation}\label{Gdel}
				G^{(\delta)}(x,y)=\left(\tfrac{1-xy}{1+(\delta-1)xy},\,\tfrac{(1-x)(1+(\delta-1)xy)}{(1+(\delta-1)x)(1-xy)}\right).
			\end{equation}
			If $(U',V')=G^{(\delta)}(X',Y')$, then
			\begin{equation}\label{u'v'}
				(U',V')\sim \mathrm{GB}_{I}(b+c,a,-a-b;\delta)\otimes \mathrm{B}_I(c,b).
			\end{equation}
		\end{theorem}
		
		\begin{remark}\label{connect} 
			Theorem \ref{AKPV} and  Theorem \ref{infin} are equivalent.  Indeed, this follows from the following three elementary observations:
			\begin{enumerate}
				\item[(a)] Let $h(x)=\tfrac x{1-x}$, which maps bijectively $(0,1)$ onto $(0,\infty)$. Let $g$ be its inverse, i.e., $g(y)=\tfrac y{1+y}$ for $y\in(0,\infty)$. Finally write $G^{(\delta)}=(G^{(\delta)}_1,\,G^{(\delta)}_2)$.
				Then 
				\begin{equation}\label{FGcon}F^{(\delta^{-1},\infty)}(x,y)=\left(\delta h\left(G_1^{(\delta)}\left(g(\tfrac{x}{\delta}),\,g(\tfrac1y)\right)\right),\,\tfrac 1{h\left(G_2^{(\delta)}\left(g(\frac x{\delta}),\,g(\frac1y)\right)\right)}\right).\end{equation}
				\item[(b)] If, for $|\lambda|<a<b$,
				$$
				(X,Y)\sim \mathrm{GB}_{II}(\lambda,a,b;\delta^{-1})\otimes \mathrm{B}_{II}(b-a,a+\lambda)
				$$ 
				then 
				$$
				(X',Y')=(g(\tfrac X{\delta}),\,g(\tfrac1Y))\sim \mathrm{GB}_{I}(b+\lambda,a-\lambda,\lambda-b;\delta)\otimes \mathrm{B}_I(a+\lambda,b-a).
				$$ 
				\item[(c)] If 
				$$
				(U',V')
                \sim \mathrm{GB}_{I}(b-\lambda,a+\lambda,-b-\lambda;\delta)\otimes\mathrm{B}_{I}(a-\lambda,b-a)
				$$
				then
				$$
				(U,V)=\left(\delta h(U'), \tfrac 1{h(V')}\right)\sim \mathrm{GB}_{II}(-\lambda,a,b;\delta^{-1})\otimes \mathrm{B}_{II}(b-a,a-\lambda).
				$$
			\end{enumerate}
		\end{remark}
		\begin{remark}\label{DR} In the special case $\delta=1$ a converse to Theorem \ref{AKPV} is known, see \cite{DarRat1971} for the case with densities and \cite{SesWes2003} for the general case. The latter reads: if $X'$ and $Y'$ are independent, $(0,1)$-valued non-Dirac random variables and $$(U',V')=G^{(1)}(X',Y')=\left(1-X'Y',\,\tfrac{1-X'}{1-X'Y'}\right)$$ are independent, then there exist positive constants $a,b,c$ such that 
			\begin{align*}
				(X',Y')&\sim\mathrm{B}_I(a+b,c)\otimes \mathrm{B}_I(a,b)
				\intertext{and}
				(U',V')&\sim \mathrm{B}_I(b+c,a)\otimes\mathrm{B}_I(c,b).   
			\end{align*}
		\end{remark}

		We note a certain analogy between: 
		\begin{itemize} \item Theorem \ref{charac1}, which is a boundary case of our main Theorem \ref{charac}, and the proof of the former
		\end{itemize}
		and 
		\begin{itemize}
			\item the characterization of the gamma and GIG (generalized inverse Gaussian) laws by the Matsumoto-Yor independence in \cite{LetWes2000}, which is a boundary case of the characterization of the GIG laws by the generalized Matsumoto-Yor property given in \cite{LetWes2024} (for which the proof of \cite{LetWes2000} serves as a  warm-up). 
		\end{itemize}
		
		In view of the connection established in \eqref{FGcon} between $F^{(1/\delta,\infty)}$, as defined in \eqref{Finf}, and $G^{(\delta)}$, as defined in \eqref{Gdel}, we have the following direct consequence of Theorem \ref{charac1}.
		\begin{corollary}
			Assume that $X'$ and $Y'$ are $(0,1)$-valued independent, non-Dirac random variables. Let $(U',V')=G^{(\delta)}(X',Y')$ for some $\delta>0$. If $U'$ and $V'$ are independent, then there exist positive $a,b,c$ such that \eqref{x'y'} and \eqref{u'v'} hold.
		\end{corollary}
		Note that this result was announced in Theorem 2.11 of \cite{KouVal2012},  which is an ``if and only if'' statement. However, no proof of the ``only if'' part was provided. Our Theorem \ref{charac1} fills the gap, in view of the equivalence between Theorems \ref{infin} and \ref{charac1}.

		\subsection{\texorpdfstring{The case $\beta=0$}{The case of beta=0}}

		To obtain a version of $F^{(\alpha,\beta)}$ for $\beta=0$, multiply the first coordinate of $F^{(\alpha,\beta)}$ by $\alpha$ and then take the reciprocal (by reciprocal of $x\neq 0$  we mean $1/x$). Further, multiply the second coordinate by $\beta$ and then take the reciprocal as well. Finally, set $\beta=0$. In this way one gets an involution  $F^{(\alpha,0)}\colon (0,\infty)^2\to(0,\infty)^2$ defined by
		\begin{equation}\label{F0}
			F^{(\alpha,0)}(x,y)=\left(\tfrac{1+x+y}{\alpha x y},\,\tfrac{1+x+y+\alpha x y}{\alpha x(1+x)}\right).
		\end{equation}

        \begin{remark}\label{equ}
The limiting cases $\beta=0$ and $\beta=\infty$ are, in fact, equivalent. Indeed, using the definitions of $F^{(\alpha,0)}$ in \eqref{F0} and $F^{(\alpha,\infty)}$ in \eqref{Finf}, we obtain
		$$
			F^{(\alpha,0)}(x,y)=\left(\tfrac 1\alpha F_1^{(\frac 1\alpha,\infty)}\left( \alpha x,\,\tfrac 1y\right),\left(F_2^{(\frac 1\alpha,\infty)}\left(\alpha x,\,\tfrac 1y\right)\right)^{-1}\right).
			$$
   \end{remark}
Thus, the map $F^{(\alpha,0)}$ is obtained from $F^{(\frac 1\alpha,\infty)}$ by the coordinate transformation
\[
(x,y)\mapsto \left(\alpha x,y^{-1}\right)
\]
and the corresponding inverse transformation on the image coordinates. So, if one of them is an IP map, then so is the other, and the corresponding independence-preserving laws can be transferred from one case to the other.

Consequently, by Theorems \ref{infin} and \ref{charac1}, we obtain the following fact.

\medskip

\begin{corollary}
Assume that $X$ and $Y$ are non-Dirac, positive random variables, which are independent. Then $(U,V)=F^{(\alpha,0)}(X,Y)$ are independent if and only if there exist parameters $a,b,\lambda$ satisfying $0<\lambda<\min\{a,b\}$ 
such that
\[
(X,Y)\sim
\mathrm{GB}_{II}(\lambda,a,b;\alpha)
\otimes
\mathrm{B}_{II}(b-\lambda,2\lambda).
\]
\end{corollary}

		\section{Proof of Theorem \ref{charac}}\label{proof}
        \subsection*{Overview of the proof}

The proof of Theorem \ref{charac} is rather long, so we begin by explaining its logic.
The argument has three main ideas.
First, the independence assumptions are translated into identities for the Laplace-type transforms $L_W^{(\gamma)}$
introduced in Section \ref{LTR}; from these identities we derive structural relations
for the auxiliary functions $M_W^{(\gamma)}$ defined in \eqref{defM}, which force a dichotomy.
Second, after excluding one of the two branches of this dichotomy, we show that the
remaining branch implies a factorization of the ratio $L_X^{(\alpha)}/L_U^{(\alpha)}$. This factorization
leads to discrete functional equations whose solutions are of Pochhammer    type, which allows to 
identify the parameters $a$, $b$ and $\lambda$, and finally determines the ratio $L_X^{(\alpha)}/L_U^{(\alpha)}$ explicitly. Third, the $L_U^{(\alpha)}$ is identified as a solution to hypergeometric difference equation. 

More precisely, the proof proceeds in the following steps.

\begin{enumerate}
    \item Subsection \ref{sec:difference}: 
    The starting point of the proof is identity \eqref{Lindep} for transforms $L_X^{(\alpha)},L_Y^{(\beta)},L_U^{(\alpha)},L_V^{(\beta)}$, which follows from the independence of $(X,Y)$ and $(U,V)=F^{(\alpha,\beta)}(X,Y)$ combined with the algebraic relations in \eqref{xyuv}. We introduce the difference-operator formalism and auxiliary functions $M_W^{(\gamma)}$ in \eqref{defM}.

\item Subsection \ref{sec:dichotomy}: We then use the difference identities satisfied by $L_W^{(\gamma)}$ to show that \eqref{Lindep} implies both the product relation \eqref{prod1} analogous to \eqref{Lindep} for the $M_W^{(\gamma)}$ and an additional linear relation \eqref{sum}. 
This leads to the dichotomy \eqref{MxMv}, \eqref{MM}: either $M_X^{(\alpha)}=M_U^{(\alpha)}$ or $\alpha' M_X^{(\alpha)}=\beta' M_V^{(\beta)}$ hold, where $\alpha'=\alpha-1$ and $\beta'=\beta-1$.

    \item Subsection \ref{sec:exclude}: Next, we prove that one branch of this dichotomy,
    \[
        \alpha' M_X^{(\alpha)}=\beta' M_V^{(\beta)},
    \]
    is impossible. The argument converts this identity into a distributional identity
    for sums of bounded random variables and then compares supports, which leads to a contradiction. Therefore the only
     possibility is
    \[
        M_X^{(\alpha)}=M_U^{(\alpha)}.
    \]

    \item Subsection \ref{sec:factorization}: Under the above identity, the ratio
    \[
        h_s(\theta,\sigma)=\frac{L_X^{(\alpha)}(\theta,\sigma,s)}{L_U^{(\alpha)}(\theta,\sigma,s)}
    \]
satisfies the multiplicative relation \eqref{4h}, which implies that $h_s$ can be written as a product of a function of $\theta$ and a function of $\sigma$.

\item Subsection \ref{sec:refined}: After comparing successive values of $s$, we refine the factorization obtained in Subsection \ref{sec:factorization} and show that $h_s$ admits the representation \eqref{LXLUN}, 
$$
h_s(\theta, \sigma)= \frac {N_1(\theta+s)N_2(\sigma)}{N_3(\theta)N_4(\sigma+s)}
$$
in terms of auxiliary functions $N_1,N_2,N_3,N_4$.

\item Subsection \ref{sec:Ni}: We then identify the $N_i$'s. The resulting discrete functional equations show that their successive quotients are affine with a common slope, so there exist parameters $a,b>0$ and $\lambda\in\mathbb R$ with $|\lambda|<\min\{a,b\}$ such that 
$$
L_X^{(\alpha)}(\theta,\sigma,s)=\frac {(b+\lambda)^{(\theta+s)}(a-\lambda)^{(\sigma)}}{(b-\lambda)^{(\sigma+s)}(a+\lambda)^{(\theta)}}L_U^{(\alpha)}(\theta,\sigma,s),
$$
see \eqref{final}. 

    \item Subsection \ref{sec:hypergeo}: After substituting the above form of $L_X^{(\alpha)}$ back into the transform identities, we obtain
    a hypergeometric difference equation (HDE) for the function
    $\theta\mapsto L_U^{(\alpha)}(\theta,0,0)$, or a related transform (see Lemma \ref{lem:HDEn}). General solutions to HDE are known and explicit. Requiring the positivity of the transforms, we conclude that
    \[
        U\sim \mathrm{GB}_{II}(-\lambda,a,b;\alpha).
    \]

    \item Subsection \ref{sec:completion}: Finally, the explicit ratio formula \eqref{final} gives
    $X\sim \mathrm{GB}_{II}(\lambda,a,b;\alpha)$. By symmetry, the same reasoning applies
    to $(Y,V)$, and a final comparison in \eqref{Lindep} shows that the same parameters
    $a,b,\lambda$ occur on both sides. This completes the proof of Theorem \ref{charac}.
\end{enumerate}
		\subsection{\texorpdfstring{Difference operators and $L_W^{(\gamma)}$}{Difference operators and LWgamma}}\label{sec:difference}
		In order to simplify the presentation it is convenient to introduce the difference operators acting on a function $f$ of two variables as follows: 
        \[
        \Delta_1 f(\theta,\sigma)=f(\theta+1,\sigma)-f(\theta,\sigma),\qquad \Delta_2 f(\theta,\sigma)=f(\theta,\sigma+1)-f(\theta,\sigma).
        \]
        Note that for two functions $g$ and $h$ of two variables, we have
		\begin{equation}\label{difprod}
			\Delta_k (gh)=\Delta_kg\,\Delta_kh+(\Delta_k g)\,h+g\,(\Delta_k h), \quad k=1,2,
		\end{equation}
which implies
		\begin{align}\label{gh}
			\Delta_1 (g\,h) \,\Delta_2 (g\,h)=g\,h\,(\Delta_1 g \,\Delta_2 h+\Delta_2 g \,\Delta_1 h)+\Delta_1 g\,\Delta_2 g \,\Delta_1 h\,\Delta_2 h\,(1+\Phi g+\Phi h)
			,
		\end{align}	
		where we define
        \[
        \Phi f=\tfrac{f\,(f+\Delta_1 f+\Delta_2 f)}{\Delta_1 f \,\Delta_2 f},
        \]
        provided the denominator is non-zero. 
        Indeed, writing $A=\Delta_1 g$, $B=\Delta_2 g$, $C=\Delta_1 h$ and $D=\Delta_2 h$, and using \eqref{difprod}, we get
\[
\Delta_1(gh)=Ah+gC+AC,\qquad \Delta_2(gh)=Bh+gD+BD.
\]
After expanding the product $\Delta_1 (g\,h) \,\Delta_2 (g\,h)$ and grouping terms, one readily singles out the expression
$gh(AD+BC)$, with the remaining terms that can be rewritten as 
\[
ABCD(1+\Phi g+\Phi h).
\]

		In the sequel, we apply the operators $\Delta_k$, $k=1,2$, to functions of three variables, with the additional argument written at the last place. Next for $s\in\R$, $\theta\ge -s$ and $\sigma\ge 0$, we can rephrase \eqref{id1} and \eqref{id2} as follows 
		\begin{equation}\label{W}
			\left\{\begin{array}{l}\Delta_1 L_W^{(\gamma)}(\theta,\sigma,s)=-\tfrac{1}{\gamma}L_W^{(\gamma)}(\theta+1,\sigma+1,s-1),\\
				\,\\ \Delta_2 L_W^{(\gamma)}(\theta,\sigma,s)=-L_W^{(\gamma)}(\theta,\sigma,s+1),\end{array}\right.
		\end{equation}
        implying
		\begin{align*}
			& L_W^{(\gamma)}(\theta,\sigma,s)+\Delta_1 L_W^{(\gamma)}(\theta,\sigma,s)+\Delta_2 L_W^{(\gamma)}(\theta,\sigma,s)\\
			&=L_W^{(\gamma)}(\theta,\sigma,s)-\tfrac{1}{\gamma}L_W^{(\gamma)}(\theta+1,\sigma+1,s-1)-L_W^{(\gamma)}(\theta,\sigma,s+1)\\
			&=\tfrac{1}{\gamma}\left(\gamma\left[L_W^{(\gamma)}(\theta,\sigma,s)-L_W^{(\gamma)}(\theta,\sigma,s+1)\right]-L_W^{(\gamma)}(\theta+1,\sigma+1,s-1)\right)\\&
			\stackrel{\eqref{id2}}{=}\tfrac{1}{\gamma}\left(\gamma L_W^{(\gamma)}(\theta,\sigma+1,s)-L_W^{(\gamma)}(\theta+1,\sigma+1,s)-L_W^{(\gamma)}(\theta+1,\sigma+2,s-1)\right)\\&
			\stackrel{\eqref{id1}}{=}\tfrac{1}{\gamma}\left(\gamma L_W^{(\gamma)}(\theta,\sigma+1,s)-L_W^{(\gamma)}(\theta+1,\sigma+1,s)
            -\gamma\left[L_W^{(\gamma)}(\theta,\sigma+1,s)-L_W^{(\gamma)}(\theta+1,\sigma+1,s)\right]\right)
            \\&=\tfrac{\gamma-1}{\gamma}\,L_W^{(\gamma)}(\theta+1,\sigma+1,s),
		\end{align*}
where in the fourth line we apply \eqref{id2} in two forms, namely
\[
L_W^{(\gamma)}(\theta,\sigma,s)-L_W^{(\gamma)}(\theta,\sigma,s+1)
=
L_W^{(\gamma)}(\theta,\sigma+1,s)
\]
and
\[
L_W^{(\gamma)}(\theta+1,\sigma+1,s-1)
=
L_W^{(\gamma)}(\theta+1,\sigma+1,s)+L_W^{(\gamma)}(\theta+1,\sigma+2,s-1),
\]
while in the fifth line we use \eqref{id1} with $\sigma$ replaced by $\sigma+1$, that is,
\[
L_W^{(\gamma)}(\theta+1,\sigma+2,s-1)
=
\gamma\bigl(L_W^{(\gamma)}(\theta,\sigma+1,s)-L_W^{(\gamma)}(\theta+1,\sigma+1,s)\bigr).
\]      
		From \eqref{W},
		\begin{equation}\label{D1D2L}
			\Delta_1 \Delta_2 L_W^{(\gamma)}(\theta,\sigma,s)=\tfrac{1}{\gamma}L_W^{(\gamma)}(\theta+1,\sigma+1,s),
		\end{equation}
		and thus
		\begin{equation}\label{ddL}
			L_W^{(\gamma)}+\Delta_1 L_W^{(\gamma)}+\Delta_2 L_W^{(\gamma)}=(\gamma-1)\Delta_1\Delta_2 L_W^{(\gamma)}.
		\end{equation}
		Consequently, the quantity 
		\begin{equation}\label{defM}
			M_W^{(\gamma)}(\theta,\sigma,s):=\frac{L_W^{(\gamma)}(\theta,\sigma,s)\,L_W^{(\gamma)}(\theta+1,\sigma+1,s)}{L_W^{(\gamma)}(\theta,\sigma,s+1)\,L_W^{(\gamma)}(\theta+1,\sigma+1,s-1)},
		\end{equation}
		thanks to \eqref{W} and \eqref{D1D2L}, is rewritten as 
		\begin{equation}\label{new}
			M_W^{(\gamma)}=\frac{L_W^{(\gamma)}\,\Delta_1\Delta_2 L_W^{(\gamma)}}{\Delta_2 L_W^{(\gamma)}\,\Delta_1 L_W^{(\gamma)}}
		\end{equation}
		and 
		\begin{equation}\label{FM}
			\Phi L_W^{(\gamma)}=(\gamma-1)M_W^{(\gamma)}.
		\end{equation}

		\subsection{\texorpdfstring{A dichotomy for the $M$-functions}{A dichotomy for the M-functions}}\label{sec:dichotomy}

		Note that \eqref{Lindep}, implies 
        \[
        				M_X^{(\alpha)}(\theta,\sigma,s)\,M_Y^{(\beta)}(\sigma,\theta,s)=	M_U^{(\alpha)}(\sigma,\theta,s)\,M_V^{(\beta)}(\theta,\sigma,s),\quad (\theta,\sigma,s)\in\Xi.
        \]
It will be convenient to work with simplified notation (note the swap of the arguments in $K_Y$ and $K_U$), for $K\in\{L,M\}$, 
		\[
		K_X(\theta,\sigma,s):=K_X^{(\alpha)}(\theta,\sigma,s),\quad K_Y(\theta,\sigma,s):=K_Y^{(\beta)}(\sigma,\theta,s),
		\]
		\[K_U(\theta,\sigma,s):=K_U^{(\alpha)}(\sigma,\theta,s),\quad K_V(\theta,\sigma,s):=K_V^{(\beta)}(\theta,\sigma,s).
		\]
Then,
		\begin{equation}\label{prod1}
			M_X\,M_Y=M_U\,M_V
		\end{equation}
		as functions defined on $\Xi$, see \eqref{setF}.  
		
		We will also prove that as functions defined on $\Xi$,
		\begin{equation}\label{sum}
			(\alpha-1)M_X+(\beta-1)M_Y=(\alpha-1)M_U+(\beta-1)M_V.
		\end{equation}
		
		In view of \eqref{W}, 
        we arrive at  
		\begin{align}
			&\left\{\begin{array}{l}\Delta_1 L_X(\theta,\sigma,s)=-\tfrac{1}{\alpha}L_X(\theta+1,\sigma+1,s-1),\\
				\Delta_2 L_X(\theta,\sigma,s)=-L_X(\theta,\sigma,s+1),\end{array}\right.\label{xxx}\\
			&\left\{\begin{array}{l}\Delta_1 L_Y(\theta,\sigma,s)=-L_Y(\theta,\sigma,s+1),\\
				\Delta_2 L_Y(\theta, \sigma,s)=-\tfrac{1}{\beta}L_Y(\theta+1,\sigma+1,s-1),\end{array}\right.\label{yyy}\\
			&\left\{\begin{array}{l}\Delta_1 L_U(\theta, \sigma,s)=-L_U(\theta, \sigma,s+1),\\
				\Delta_2 L_U(\theta,\sigma,s)=-\tfrac{1}{\alpha}L_U(\theta+1,\sigma+1,s-1),\end{array}\right.\label{uuu}\\
			&\left\{\begin{array}{l}\Delta_1 L_V(\theta,\sigma,s)=-\tfrac{1}{\beta}L_V(\theta+1,\sigma+1,s-1),\\
				\Delta_2 L_V(\theta,\sigma,s)=-L_V(\theta,\sigma,s+1),\end{array}\right.\label{vvv}
		\end{align}
        which hold for $(\theta,\sigma,s)\in\Xi$. 
		In view of \eqref{Lindep}, the above relations yield the following identities of functions defined on $\Xi$, namely $$\Delta_1 L_X\;\Delta_2 L_Y= \Delta_2 L_U\;\Delta_1 L_V\quad\mbox{and}\quad \Delta_2 L_X\;\Delta_1 L_Y= \Delta_1 L_U\;\Delta_2L_V,
		$$
		and, summing and multiplying the two equality relations side-by-side, 
		\begin{equation}\label{del1}
			\Delta_1 L_X\;\Delta_2 L_Y+\Delta_2 L_X\;\Delta_1L_Y=\Delta_1 L_U\;\Delta_2 L_V+\Delta_2 L_U\;\Delta_1 L_V
		\end{equation}
		and
		\begin{equation}\label{del2}
			\Delta_1 L_X\;\Delta_2 L_X\;\Delta_1 L_Y\;\Delta_2 L_Y=\Delta_1 L_U\;\Delta_2 L_U\;\Delta_1 L_V\;\Delta_2 L_V.
		\end{equation}
		Finally, \eqref{Lindep} implies $\Delta_{k}\left(L_X\,L_Y\right)=\Delta_{k}\left(L_U\,L_V\right)$, for $k=1,2$. 
		Therefore, in view of \eqref{gh}, applied to $(g,\,h)=\left(L_X,\,L_Y\right)$ and to $(g,\,h)=\left(L_U,\,L_V\right)$, and after cancellations made possible by \eqref{del1} and \eqref{del2}), we get the equality of functions defined on $\Xi$,
		\begin{equation*}
			\Phi L_X+\Phi L_Y=\Phi L_U+\Phi L_V.
		\end{equation*}
		Taking \eqref{FM} into account, we obtain \eqref{sum}.
		
		Via simple algebra from \eqref{prod1} and \eqref{sum} we get  on the set  $\Xi$,
		\begin{equation}\label{sym1}
			(\alpha' M_X-\beta' M_V)(M_X-M_U)=0
		\end{equation}
		and
		\begin{equation}\label{sym2}
			(\beta' M_Y-\alpha' M_U)(M_Y-M_V)=0
		\end{equation}
		with $\alpha'=\alpha-1$ and $\beta'=\beta-1$.

		We will address only the first of these two equations. The second can be handled in a similar manner. Equation \eqref{sym1} implies that either
		 \begin{equation}\label{MxMv}\alpha' M_X=\beta' M_V\qquad\mbox{on }\Xi \end{equation} 
         or
         \begin{equation}\label{MM}M_X=M_U\qquad\mbox{on }\Xi.\end{equation}
To justify this, we introduce the open set
\[
D=\{(\theta,\sigma,s)\in\mathbb C^3\colon \ \Re(\theta)>0,\ \Re(\sigma)>0, \, \Re(\theta+s)>0,\ \Re(\sigma+s)>0\}.
\]
For any fixed $\gamma>0$, the transform $L_W^{(\gamma)}$ extends holomorphically to $D$. Consequently, the functions $M_X$, $M_U$ and $M_V$, being quotients of products of such transforms, are meromorphic on $D$. Hence $F := \alpha' M_X-\beta' M_V$ and $G:=M_X-M_U$ are meromorphic on $D$ as well. 
By \eqref{sym1}, we have $FG=0$ on $D\cap\mathbb R^3$. Therefore, after multiplying by a common denominator, we obtain a holomorphic identity on $D$ which vanishes on the nonempty real open subset $D\cap\mathbb R^3$. Hence, by the identity theorem for holomorphic functions, this identity holds on all of $D$, and so $FG=0$ on $D$. Since the ring of meromorphic functions on a connected domain has no zero divisors (see the reasoning of Section 4.2 in \cite{LetWes2024}), we conclude that either $F=0$ on $D$ or $G=0$ on $D$. In other words, either \eqref{MxMv} or \eqref{MM} holds on $D$.
Finally, since $\Xi = \overline{D\cap\R^3}$, using continuity of the functions involved, we conclude that either \eqref{MxMv} or \eqref{MM} holds on $\Xi$.
		
		\subsection{\texorpdfstring{Excluding the branch $\alpha' M_X=\beta' M_V$}{Excluding the branch alpha MX=beta MV}}\label{sec:exclude}
		Observe that, since $\alpha\neq\beta$, the above relation implies $\alpha'\beta'\neq 0$ and denote  
		$$
		C(s):=\tfrac{L_X(1,1,s)\,L_V(0,0,s+1)}{L_X(0,0,s+1)\,L_V(1,1,s)}.
		$$
        
		Then, \eqref{MxMv} for $\theta=\sigma=0$ gives
		$$
		\alpha' C(s)=\beta' C(s-1),\quad s\in\mathbb N.
		$$
		Iterating with respect to $s=1,2,\ldots$, we see that 
		$$
		C(s)=\kappa^s\,C(0),\quad \mbox{where }\;\kappa:=\tfrac{\beta'}{\alpha'}\neq 1\mbox{ and }\kappa>0.
		$$
		Thus
		$$
		\tfrac{L_X(1,1,s)\,L_V(0,0,s+1)}{L_X(0,0,s+1)\,L_V(1,1,s)}=\,\kappa^s\,\tfrac{L_X(1,1,0)\,L_V(0,0,1)}{L_X(0,0,1)\,L_V(1,1,0)},
		$$
		which can be rewritten as 
		$$
		\tfrac{L_X(1,1,s)}{L_X(1,1,0)}\,\tfrac{L_V(0,0,s+1)}{L_V(0,0,1)}=\kappa^s\,\tfrac{L_X(0,0,s+1)}{L_X(0,0,1)}\,\tfrac{{\,L_V(1,1,s)}}{\,L_V(1,1,0)}.
		$$
		Then for mutually independent $X_0, X_1, V_0, V_1$, with
		\[
		\widetilde X_i\sim \frac{\tfrac{\alpha^i\,x}{(1+x)(1+\alpha x)^i}\P_X(dx)}{L_X(i,i,1-i)} \quad\mbox{and}\quad \widetilde V_i\sim \frac{\tfrac{\beta^i\,v}{(1+v)(1+\beta v)^i}\,\P_V(dv)}{L_V(i,i,1-i)}, \quad i=0,1,
		\]
		with
		\[
		X_i:=\log\,\tfrac{\widetilde X_i}{1+\widetilde X_i}\quad\mbox{and}\quad V_i:=\log\,\tfrac{\widetilde V_i}{1+\widetilde V_i},\quad i=0,1,
		\]
		we get
		\begin{equation}\label{???}
			\E\left[e^{s(X_1+V_0)}\right]=\E\left[e^{s \left(c+X_0+V_1\right)}\right],\quad s=0,1,2,\ldots,
		\end{equation}
		where  $c=\log\,\kappa\neq 0$.
		
		Note that \eqref{???} implies $X_1+V_0\stackrel{d}{=}c+X_0+V_1$. Indeed, $e^{X_i}$ and $e^{V_i}$, $i=0,1$, are $[0,1]$-valued random variables, so all distributions are determined by their moments. Consequently,  \begin{equation}\label{supp}
			\mathrm{supp}(\P_{X_1}*\P_{V_0})
			=c+\mathrm{supp}(\P_{X_0}*\P_{V_1}),
		\end{equation}
		which turns out to be impossible. Indeed, first note that $$\mathrm{supp}(\P_{X_0}*\P_{V_1})=\mathrm{supp}(\P_{X_1}*\P_{V_0}),$$
		since $$\mathrm{supp}(\P_{\widetilde X_0})=\mathrm{supp}(\P_{\widetilde X_1})=\mathrm{supp}(\P_X)\subset[0,\infty)$$ and $$\mathrm{supp}(\P_{\widetilde V_0})=\mathrm{supp}(\P_{\widetilde V_1})=\mathrm{supp}(\P_V)\subset[0,\infty).$$  
		By taking into account how the transformation $\log\tfrac{x}{1+x}$ acts on $\mathrm{supp}(\P_X)$ and $\mathrm{supp}(\P_V)$, we conclude that the supremum $u$ of  $\mathrm{supp}(\P_{X_0}*\P_{V_1})=\mathrm{supp}(\P_{X_1}*\P_{V_0})$ is a finite number. Thus, \eqref{supp} would imply $u=c+u$, which is impossible, since necessarily $c\neq 0$.
		
		\subsection{\texorpdfstring{Factorization of the ratio $L_X/L_U$}{Factorization of the ratio LX/LU}}\label{sec:factorization}
		\begin{lemma}\label{lem:hAB}
			Let $\mathbb N_0^2\ni(\theta,\sigma) \mapsto h(\theta,\sigma)\in (0,\infty)$. Then, there exist two functions $A$ and $B$ defined on $\mathbb{N}_0$ such that $h(\theta,\sigma)=A(\theta)B(\sigma)$ if and only if, for all $(\theta,\sigma)$, we have \begin{equation}\label{4h} 
				h(\theta,\sigma)h(\theta+1,\sigma+1)=h(\theta+1,\sigma)h(\theta,\sigma+1).
			\end{equation}
		\end{lemma}
		\begin{proof}
			$\Rightarrow$ is clear.
			
			$\Leftarrow$: For any fixed $\theta \in \mathbb {N}_0$, observe that 
            $$
            \mathbb N_0\ni \sigma\mapsto \frac{h(\theta +1,\sigma)}{h(\theta ,\sigma)}= \frac{h(\theta+1,\sigma+1)}{h(\theta,\sigma+1)}:=a(\theta)
            $$ 
            is a constant function in $\sigma$. Therefore, for $\theta \in \mathbb {N}$,
			$$
			h(\theta,\sigma)=a(\theta-1)h(\theta-1,\sigma)
			$$
			and iterating this relation,
			$$h(\theta,\sigma)=a(\theta-1)a(\theta-2)\cdots a(0)h(0,\sigma).$$ With $A(\theta):=a(\theta-1)a(\theta-2)\cdots a(0)$ for $\theta>0$, $A(0):=1$, and $B(\sigma)=h(0,\sigma)$, the lemma is proved.
		\end{proof}
		\begin{lemma}\label{lem:hs}
			Assume \eqref{MM}.   Then, for any fixed $s \in \mathbb {N}_0$, the function 
			\begin{equation}\label{forgotten} 
				h_s(\theta, \sigma):=\frac {L_X(\theta, \sigma,s)}{L_U(\theta, \sigma,s)},
			\end{equation}
			defined on $\mathbb {N}^2_0$, satisfies \eqref{4h}.
		\end{lemma}
		\begin{proof}
			

            Referring to \eqref{new}, we see that \eqref{MM} implies that on $\Xi$
			\begin{equation}\label{fourterms}
				L_X \,\Delta_1 \Delta_2 L_X\,\Delta_1 L_U \,\Delta_2 L_U=L_U \,\Delta_1\Delta_2 L_U\,\Delta_1 L_X \,\Delta_2 L_X.
			\end{equation}
			
			First, consider the case $\alpha\neq 1$. Plugging \eqref{ddL} (for both $L_X$ and $L_U$) into \eqref{fourterms}, and canceling $(\alpha-1)$, we get
			\begin{align}\label{twiceasnew}
				L_X (L_X+\Delta_1 L_X+\Delta_2 L_X)\Delta_1 L_U \,\Delta_2 L_U
				=L_U (L_U+\Delta_1 L_U+\Delta_2 L_U)\Delta_1 L_X \,\Delta_2 L_X.
			\end{align}
			Adding to both sides the term $\Delta_1 L_X \Delta_2 L_X \Delta_1 L_U \Delta_2 L_U$, we get
			\begin{align*}
				\Delta_1 L_U \Delta_2 L_U (\Delta_1 L_X+ L_X)(\Delta_2 L_X+ L_X) = \Delta_1 L_X \Delta_2 L_X (\Delta_1 L_U+ L_U)(\Delta_2 L_U+ L_U).
			\end{align*}
			Therefore,
			$$
			\frac {(\Delta_1 L_X +L_X)(\Delta_2 L_X +L_X)}{(\Delta_1 L_U +L_U)(\Delta_2 L_U +L_U)}=\frac {\Delta_1 L_X\Delta_2 L_X }{\Delta_1 L_U\Delta_2 L_U}=\frac {L_X \Delta_1 \Delta_2 L_X}{L_U \Delta_1 \Delta_2 L_U},
			$$ where the last equality follows from \eqref{MM}. By referring to \eqref{D1D2L}, we thus arrive at
			\begin{equation}\label{factorization}
				\frac{L_X(\theta+1,\sigma,s)L_X(\theta,\sigma+1,s)}{L_U(\theta+1,\sigma,s)L_U(\theta,\sigma+1,s)}=\frac{L_X(\theta,\sigma,s)L_X(\theta+1,\sigma+1,s)}{L_U(\theta,\sigma,s)L_U(\theta+1,\sigma+1,s)},
			\end{equation}
			i.e., the relation \eqref{4h} for the function $h_s$ defined in \eqref{forgotten}, as required.
			
			In the case $\alpha=1$, using the obvious identities
			$$
L_X(\theta,\sigma,s)=L_X(\theta+s,\sigma,0), \qquad 			L_{U}(\theta, \sigma,s)=L_{U}(\theta, \sigma+s,0),
			$$
  we establish \eqref{factorization} directly. 
		\end{proof}
		
		As a consequence of Lemma \ref{lem:hAB} and Lemma \ref{lem:hs}, for any $s\in\mathbb{N}_0$, there exist functions $a_s$ and $b_s$ such that for $\theta, \sigma \in\mathbb{N}_0$,
		\begin{equation}\label{LabL}
			h_s(\theta,\sigma)=\frac{L_X(\theta,\sigma,s)}{L_U( \theta,\sigma, s)}=a_s(\theta)b_s(\sigma).
		\end{equation}
		
\subsection{\texorpdfstring{A refined representation of $L_X/L_U$}{A refined representation of LX/LU}}\label{sec:refined}
        
		Next, to provide more insight into the functions $a_s$ and $b_s$, $s \in \mathbb N_0$,  let us introduce a function
\begin{equation}\label{referee}
T_s(\theta,\sigma) := \frac{h_{s-1}(\theta+1,\sigma+1)}{h_s(\theta,\sigma)}=\frac{a_{s-1}(\theta+1)b_{s-1}(\sigma+1)}{a_s(\theta)b_s(\sigma)},\quad s\in\mathbb N.
\end{equation}
        Rewriting it as
		\begin{equation}\label{defT}
			T_s(\theta,\sigma)=\tfrac{L_X(\theta+1,\sigma+1,s-1)L_U(\theta,\sigma,s)}{L_X(\theta,\sigma,s)L_U(\theta+1,\sigma+1,s-1)}.
		\end{equation}
		thanks to \eqref{MM}, it is seen that
		$$
		T_s(\theta,\sigma)=T_{s+1}(\theta,\sigma),\quad (\theta,\sigma)\in\mathbb {N}_0^2,
		$$
		which upon iteration yields
		\begin{equation}\label{T0}
			T_s(\theta,\sigma)=T_1(\theta,\sigma),\quad s\in\mathbb N.
		\end{equation}
		Referring to \eqref{LabL} for $T_{s+1}$ and using \eqref{T0} we obtain
		\begin{equation}\label{astt}
			T_{s+1}(\theta,\sigma)=\frac{a_{s}(\theta+1)b_{s}(\sigma+1)}{a_{s+1}(\theta)b_{s+1}(\sigma)}=a(\theta)\,b(\sigma),\quad s\in\mathbb N_0,
		\end{equation}
		where 
		$$a(\theta)=\frac{a_0(\theta+1)}{a_1(\theta)},\quad b(\sigma)=\frac{b_0(\sigma+1)}{b_1(\sigma)}.$$
		From \eqref{astt}, first with $\sigma=0$ and then with $\theta=0$,  we get
		\begin{equation}\label{ab1}
			\frac{a_s(\theta+1)}{a_{s+1}(\theta)}=a(\theta)K(s)\quad\mbox{and}\quad \frac{b_s(\sigma+1)}{b_{s+1}(\sigma)}=\frac{b(\sigma)}{K(s)},
		\end{equation}
		where 
		$$
		K(s)=\frac{b(0)b_{s+1}(0)}{b_s(1)}=\frac{a_s(1)}{a(0)a_{s+1}(0)}
		$$
		is obtained by setting $\theta=\sigma=0$ in \eqref{ab1}.
		
		From \eqref{ab1}, for $s \in \mathbb N$, we have for $\theta \in \mathbb N$,
		\begin{align*}
			a_{s+1}(\theta)=&\tfrac {a_s(\theta+1)}{K(s)a(\theta)}=
			\tfrac {a_{s-1}(\theta+2)}{K(s)K(s-1)a(\theta)a(\theta+1)}\\
			=&\tfrac {a_0(\theta+s+1)}{K(s)\cdots K(0)\,a(\theta)\cdots a(\theta+s)}
			=\tfrac {a_0(\theta+s+1)\,a(0)\cdots a(\theta-1)}{K(s)\cdots K(0)\,a(0)\cdots a(\theta+s)}
		\end{align*}
		and for $\sigma \in \mathbb N$,
		\begin{align*}
			b_{s+1}(\sigma)=& \tfrac{K(s)b_s(\sigma+1)}{b(\sigma)}=\tfrac{K(s)K(s-1) b_{s-1}(\sigma+2)}{b(\sigma)b(\sigma+1)}\\
			=&\tfrac{K(s)\cdots K(0) b_0(\sigma+s+1)}{b(\sigma) \cdots b(\sigma+s)}=\tfrac {K(s)\cdots K(0)\,b(0) \cdots b(\sigma-1)b_0(\sigma+s+1)}{b(0) \cdots b(\sigma+s)}.
		\end{align*}
		
		As a consequence, replacing $s+1$ by $s$ in the previous formulas,and interpreting empty products as \(1\), we obtain, after canceling the $K$ terms,
		$$
		h_s(\theta,\sigma)=a_s(\theta)b_s(\sigma)
		=\frac {a_0(\theta+s)\,a(0)\cdots a(\theta-1)}{a(0)\cdots a(\theta+s-1)}\,\frac {b(0) \cdots b(\sigma-1)b_0(\sigma+s)}{b(0) \cdots b(\sigma+s-1)}
		$$
		for $\theta,\sigma,s\in \mathbb N_0$. Referring back to \eqref{LabL}, we can rewrite the above display as
		\begin{equation}\label{LXLUN}
			L_X(\theta, \sigma,s)=\frac {N_1(\theta+s)N_2(\sigma)} {N_3(\theta) N_4 (\sigma+s)}\,
			L_U(\theta, \sigma,s),
		\end{equation}
		where
        \begin{gather*}
N_1(z) \propto \frac {a_0(z)}{a(0)\cdots a(z-1)},\qquad N_2(z) \propto b(0) \cdots b(z-1),\\
N_3(z) \propto \frac{1}{a(0)\cdots a(z-1)},\qquad N_4(z) \propto \frac {b(0) \cdots b(z-1)}{b_0(z)}.
        \end{gather*}
        Since both $L_X$ and $L_U$ are positive functions equal to $1$ at $(0, 0, 0)$, we can assume $N_i(0)=1$ and $N_i(z) >0$, for $z \in \mathbb N$ and $i=1, 2, 3, 4$. 
		
	We will also need the following shifted version of \eqref{LXLUN}:
\begin{equation}\label{LXLUNshift}
L_X(\theta+1,\sigma+1,-1)=
\frac{N_1(\theta)N_2(\sigma+1)}
{N_3(\theta+1)N_4(\sigma)}L_U(\theta+1,\sigma+1,-1),
\qquad \theta,\sigma\in\mathbb N_0.
\end{equation}
Indeed, by \eqref{MM} at \(s=0\),
\[
\frac{L_X(\theta,\sigma,0)L_X(\theta+1,\sigma+1,0)}
{L_X(\theta,\sigma,1)L_X(\theta+1,\sigma+1,-1)}
=
\frac{L_U(\theta,\sigma,0)L_U(\theta+1,\sigma+1,0)}
{L_U(\theta,\sigma,1)L_U(\theta+1,\sigma+1,-1)}.
\]
Rearranging and using \eqref{LXLUN} for \(s=0\) and \(s=1\), we obtain \eqref{LXLUNshift}.
		\subsection{\texorpdfstring{Identification of $N_i$ and of the parameters $a,b,\lambda$}{Identification of Ni and of the parameters a,b,lambda}}\label{sec:Ni}
		
		Next, we identify the functions $N_i$, $i=1,2,3,4$.
		Applying $\Delta_1$ to \eqref{LXLUN}, we get
		$$
		\Delta_1 L_X (\theta, \sigma,s)= \tfrac {N_2(\sigma)}{N_4(\sigma+s)} \left[\tfrac {N_1(\theta+s+1)}{N_3(\theta+1)}L_U(\theta+1,\sigma,s)-\tfrac {N_1(\theta+s)}{N_3(\theta)} L_U(\theta, \sigma, s)\right].
		$$
		Referring again to \eqref{LXLUN} (in case  \(s=0\) we use  \eqref{LXLUNshift} instead), we see that 
		$$
		-\tfrac1\alpha\,L_X (\theta+1,\sigma+1,s-1)=-\tfrac1\alpha\,\tfrac {N_1(\theta+s)N_2(\sigma+1)}{N_3(\theta+1)N_4(\sigma+s)}\,L_U (\theta+1, \sigma+1,s-1),\quad s\in\mathbb N_0.
		$$
        Since the left-hand sides of the  two above identities are equal by the first identity of \eqref{xxx},
		equating the right-hand sides, and multiplying both sides by 
        $$
        \frac {N_3(\theta+1)N_4(\sigma+s)}{N_1(\theta+s)N_2(\sigma)},
        $$
		we obtain
		\begin{align}\label{aformula}
			n_3(\theta)L_U(\theta, \sigma,s) -n_1(\theta+s)L_U(\theta+1,\sigma,s)
			= \tfrac{n_2(\sigma)}\alpha\,L_U(\theta+1, \sigma+1,s-1)
		\end{align}
		with the notation 
        \begin{equation}\label{ratios}
        n_i(z)=\frac {N_i(z+1)}{N_i(z)}, \qquad i=1,2,3,4.
        \end{equation} 
        In an analogous way, by applying $\Delta_2$ to \eqref{LXLUN}, we get
		$$
		\Delta_2 L_X (\theta, \sigma,s)= \tfrac {N_1(\theta+s)}{N_3(\theta)} \left[\tfrac {N_2(\sigma+1)}{N_4(\sigma+s+1)}L_U(\theta,\sigma+1,s)-\tfrac {N_2(\sigma)}{N_4(\sigma+s)} L_U(\theta, \sigma,s)\right].
		$$
		Referring again to \eqref{LXLUN} we get
		$$
		-L_X (\theta,\sigma,s+1)=-\frac {N_1(\theta+s+1)N_2(\sigma)}{N_3(\theta)N_4(\sigma+s+1)}L_U (\theta, \sigma,s+1),\quad s\in\mathbb N_0.
		$$
        Since the left-hand sides of the two last identities are equal by the second identity of \eqref{xxx}, equating the right-hand sides and multiplying both sides by
		$$
		\frac{N_3(\theta)N_4(\sigma+s+1)}{N_1(\theta+s)N_2(\sigma)},
		$$
		we obtain 
		\begin{align}\label{aformula2}
			n_2(\sigma) L_U(\theta,\sigma+1,s)-n_4(\sigma+s)L_U(\theta,\sigma,s)
			=-n_1(\theta+s) L_U(\theta, \sigma,s+1).
		\end{align}
		
		Subtracting \eqref{aformula2} from \eqref{aformula}  side-wise and using both identities of \eqref{uuu}  we obtain
\[
\bigl(n_3(\theta)-n_1(\theta+s)+n_4(\sigma+s)-n_2(\sigma)\bigr)L_U(\theta,\sigma,s)=0.
\]
Dividing by $L_U(\theta,\sigma,s)$, we get, for $\theta,\sigma, s\in\mathbb N_0$,
\begin{equation}\label{best}
n_1(\theta+s)-n_3(\theta)=n_4(\sigma+s)-n_2(\sigma)=f(s),
\end{equation}
for some function $f$, where the last equality follows from the principle of separation of variables. This is exactly the Cauchy-Pexider equation we already dealt with in the proof of Theorem \ref{charac1}, see \eqref{CPE}. Thus, we conclude that $f$ and $n_i$, $i=1,2,3,4$, are affine functions with a common slope $q\in \mathbb R$.

		We will show that $q\neq 0$. Assume that $q=0$.  Then, recalling \eqref{ratios} and $N_i(0)=1$, we deduce that $N_i(s)=a_i^s$, 
		$s=0,1,\ldots$, where $a_i$, $i=1,2,3,4$, are positive constants satisfying $a_1-a_3=a_4-a_2$. Thus, \eqref{LXLUN} yields
		\begin{align*}
			\E\left[\ \left(\tfrac{\alpha X}{1+\alpha X}\right)^\theta\,\left(\tfrac1{1+X}\right)^\sigma\, \left(\tfrac X{1+X}\right)^s\right]
			=\left(\tfrac{a_2}{a_4}\right)^\sigma\left(\tfrac{a_1}{a_3}\right)^\theta \, \left(\tfrac{a_1}{a_4} \right)^s \E\left[\left(\tfrac1{1+U}\right)^\theta\,\left(\tfrac{\alpha U}{1+\alpha U}\right)^\sigma \left(\tfrac{U}{1+U}\right)^s\,\right].
		\end{align*}
		Consequently, we have
		$$
		\left(\,\tfrac{\alpha X}{1+\alpha X},\,\tfrac1{1+X},\,\tfrac X{1+X}\right)\stackrel{d}{=}\left(\tfrac{a_1}{a_3}\,\tfrac1{1+U},\,\tfrac{a_2}{a_4}\tfrac{\alpha U}{1+\alpha U}, \tfrac{a_1}{a_4}\,\tfrac{U}{1+U},\right),
		$$
		whence $$\tfrac{a_4}{a_1}\,\tfrac{X}{1+X}+\tfrac{a_3}{a_1}\tfrac{\alpha X}{1+\alpha X}=1.$$
		Equivalently,
		$$
		\alpha a_2 X^2+\{\alpha(a_2-a_4)+a_4-a_1\}X-a_1=0,
		$$
		and since $a_1,a_2>0$ and $X$ is positive and non-Dirac, this is impossible. Consequently $q\neq 0$. 
        
        Therefore, we can write $a_i=q\tilde a_i$ which gives $n_i(z)=q(z+\tilde a_i)$, $i=1,2,3,4$. Then each $N_i(z)$ contains a factor $q^z$, and these factors cancel in the ratio appearing in \eqref{LXLUN}. Thus the value of $q$ does not affect \eqref{LXLUN}, and without loss of generality we may take $q=1$.
		
		Setting
		$$
		\lambda= \tfrac {n_3(0)-n_2(0)}{2},\,\, a=\tfrac {n_3(0)+n_2(0)}{2}, \,\, b=f(0)+\tfrac {n_3(0)+n_2(0)}{2},
		$$
		we get $f(z)=z+b-a$, and 
		\begin{align}
			\begin{split}\label{n1234}
				n_1(z)=z+b+\lambda, & \quad n_2(z)=z+a-\lambda,\\ 
				n_3(z)=z+a+\lambda, & \quad n_4(z)=z+b-\lambda.
			\end{split}
		\end{align}
   Since each function $n_i$ is positive on $\mathbb{N}_0$, we have $n_i(0)>0$, i.e. $b\pm\lambda>0$ and $a\pm\lambda>0$, which is equivalent to $|\lambda|<\min\{a,b\}$.
		
		Observing that \eqref{ratios} implies that 
        $$
        N_i(z)=n_i (z-1)\cdot \cdots \cdot n_i(0) \cdot N_i(0),\qquad i=1,2,3,4
        $$
        and recalling that $N_i(0)=1$, for $i=1,2,3,4$, we end up with
		\begin{align*}
			N_1(z)=(b+\lambda)^{(z)}, & \quad
			N_2(z)=(a-\lambda)^{(z)}, \\
			N_3(z)=(a+\lambda)^{(z)}, & \quad 
			N_4(z)=(b-\lambda)^{(z)},
		\end{align*}
		and thus, for \(\theta,\sigma,s\in\mathbb N_0\), 
		\begin{equation}\label{final}
			h_s(\theta, \sigma)=\frac {L_X(\theta, \sigma,s)}{L_U(\theta, \sigma,s)}=\frac {(b+\lambda)^{(\theta+s)}(a-\lambda)^{(\sigma)}}{(b-\lambda)^{(\sigma+s)}(a+\lambda)^{(\theta)}}.
		\end{equation}
		
		\subsection{\texorpdfstring{A hypergeometric difference equation for $L_U$}{A hypergeometric difference equation for LU}}\label{sec:hypergeo}
		
		\begin{lemma} For $(\theta,\sigma,s) \in\mathbb{N}_0^3$ 
        \begin{multline}\label{explicit}
        \left\{\sigma+(2\alpha-1)\theta-(1-\alpha)(1+2\lambda)+\alpha(s+a+b)
    \right\}L_U(\theta+1,\sigma,s)\\=\alpha(\theta+a+\lambda)L_U(\theta,\sigma,s)-(1-\alpha)(\theta+1+s+b+\lambda)L_U(\theta+2,\sigma,s).
        \end{multline}
			
		\end{lemma}
		\begin{proof}
			From identity \eqref{aformula2}, by rewriting $n_1(\theta+s)$ as $n_3(\theta)-n_2(\sigma)+n_4(\sigma+s)$, and using the first identity in \eqref{uuu}, we obtain
			\begin{equation}\label{nnn}
				n_4(\sigma+s)\,L_U(\theta+1,\sigma,s) =(n_3(\theta)-n_2(\sigma))L_U(\theta,\sigma,s+1)+n_2(\sigma)L_U(\theta,\sigma+1,s).
			\end{equation}
			Notice that subtracting sidewise the second from the first of equations of \eqref{uuu}  (the second taken with $s+1$ replacing $s$ and the first taken with $\sigma+1$ replacing $\sigma$) one gets the identity
			$$
			L_U(\theta,\sigma,s+1)-L_U(\theta,\sigma+1,s)=\tfrac{1-\alpha}{\alpha}\,L_U(\theta+1,\sigma+1,s),
			$$
			that, inserted in \eqref{nnn}, yields
			\begin{align}\label{n4s}
				n_4(\sigma+s)\,L_U(\theta+1,\sigma,s) =n_3(\theta)L_U(\theta,\sigma,s+1)-\tfrac{1-\alpha}{\alpha}\,n_2(\sigma)L_U(\theta+1,\sigma+1,s).
			\end{align}
			Again referring to the first identity of \eqref{uuu}  we arrive at
			\begin{multline}
				(n_4(\sigma+s)+\alpha n_3(\theta))\,L_U(\theta+1,\sigma,s) =
				\alpha n_3(\theta)L_U(\theta,\sigma,s)\\+(1-\alpha)\left\{n_3(\theta)L_U(\theta,\sigma,s+1)-\tfrac 1{\alpha}n_2(\sigma)L_U(\theta+1,\sigma+1,s)\right \}.
			\end{multline}
			Now we refer to \eqref{aformula}, getting
			\begin{multline}
				(n_4(\sigma+s)+\alpha n_3(\theta))\,L_U(\theta+1,\sigma,s) =
				\alpha n_3(\theta)L_U(\theta,\sigma,s)\\+(1-\alpha)n_1(\theta+1+s)L_U(\theta+1,\sigma,s+1),
			\end{multline}
			that, again by the first identity of \eqref{uuu}, written with $\theta+1$ replacing $\theta$, is rewritten as
        \begin{multline}\label{R15a}
				\left\{n_4(\sigma+s)+\alpha n_3(\theta)-(1-\alpha)n_1(\theta+1+s)\right\}\,L_U(\theta+1,\sigma,s)\\=
				\alpha n_3(\theta)L_U(\theta,\sigma,s)-(1-\alpha)n_1(\theta+1+s)L_U(\theta+2,\sigma,s),
			\end{multline}
            which is rewritten
            as in \eqref{explicit}, once the expressions for the $n_i$'s in \eqref{n1234} are substituted.
		\end{proof}

		The difference equation \eqref{R15a} in $\theta$ for fixed $(\sigma,s)$ is a second-order linear difference equation with affine coefficients. Equations of this kind were named as hypergeometric difference equations in \cite[Chapter III]{HDE}, where
        they were defined on a region of the complex plane.
        Each such equation can be transformed into its normal form, which is introduced below; cf. \cite[(99) and (103)]{HDE}. 
		\begin{lemma}\label{lem:hde}\ 
			\begin{enumerate}
				\item[(i)] If $\alpha\neq 1$, the function $\ell(x)=L_U(x-\beta_3,0,0)$, defined for $x \in \beta_3+\mathbb {N}_0$, satisfies the hypergeometric difference equation of the form
				\begin{multline}\label{hde}
					(x+2+\beta_1+\beta_2)\ell(x+2)\\
					-\left \{(1+\rho)(x+1)+\rho\beta_1+\beta_2 \right\}\ell(x+1)
					+\rho\, x\,\ell (x)=0,
				\end{multline}
				where \begin{equation}\label{parameters}\rho=\frac{\alpha}{\alpha-1}, \qquad  \beta_1=b-\lambda-1>-1, \qquad \beta_2=\lambda-a<0, \qquad \beta_3=\lambda+a>0.
                \end{equation}
				\item[(ii)] If $\alpha=1$, then for $\theta \in\mathbb{N}_0$,
				\[
				L_U(\theta,0,0) = \tfrac{B(b-\lambda,a+\lambda+\theta)}{B(b-\lambda,a+\lambda)}.
				\]
			\end{enumerate}
		\end{lemma}
		\begin{proof}
			(i)   Equation \eqref{explicit} for $s=\sigma=0$ is 
			\begin{multline*}
				(1-\alpha)(\theta+b+\lambda+1)L_U(\theta+2,0,0)\\
				+\left \{(2\alpha-1)\theta+\alpha(a+b)-(1-\alpha)(1+2\lambda)\right \}L_U(\theta+1,0,0)\\-\alpha(\theta+a+\lambda) L_U(\theta,0,0)=0,\qquad \theta \in\mathbb{N}_0.
			\end{multline*}
			which by dividing for $1-\alpha$ is easily seen to have the form 
            \begin{multline*}
(\theta+2+\beta_1+\beta_2+\beta_3)L_U(\theta+2,0,0)\\-\left \{ (1+\rho)(\theta+1+\beta_3)+\beta_1\rho+\beta_2 \right \}L_U(\theta+1,0,0)+\rho(\theta+\beta_3)L_U(\theta,0,0)=0,
            \end{multline*}
            turned into the form \eqref{hde}, once $\ell(\theta+\beta_3)=L_U(\theta,0,0)$ is introduced, for $\theta \in \mathbb N_0$, and the replacement $x=\theta+\beta_3$ is performed.
			
			(ii) For $\alpha=1$ we obtain a first order difference equation of the form 
			\begin{align*}
				\left(\theta+a+b\right)L_U(\theta+1,0,0)-(\theta+a+\lambda) L_U(\theta,0,0)=0,\qquad \theta \in\mathbb{N}_0.
			\end{align*}
			Under the initial condition $L_U(0,0,0)=1$, it has a unique solution
			\[
			L_U(\theta,0,0) = \tfrac{(a+\lambda)^{(\theta)}}{(a+b)^{(\theta)}}=\tfrac{\Gamma(a+b)\Gamma(a+\lambda+\theta)}{\Gamma(a+b+\theta)\Gamma(a+\lambda)}=
            \tfrac{B(b-\lambda,a+\lambda+\theta)}{B(b-\lambda,a+\lambda)}. 
			\]
		\end{proof}

We refer to $(\rho ,\beta_1,\beta_2,\beta_3)$ as the parameters of \eqref{hde}. For the sequel we introduce the $\rho$-difference operator
\begin{equation}\label{rhodifference}
\delta _{\rho}f(x)=f(x+1)-\rho f(x),
\end{equation}
that for $\rho=1$, becomes the standard difference operator.

        \begin{lemma}\label{lem:HDEn}\ 
			With $\alpha\neq 1$, the sequence $(\ell^{(n)}, n \in \mathbb N)$ defined by $\ell^{(1)}=\ell$ and $\ell^{(n)}=\delta_ {\rho}^{n-1}\ell$, for $n=2,3,\ldots$ has the following properties: 
			\begin{enumerate}
				\item[(i)] 
				 The function $\ell^{(n)}$ satisfies \eqref{hde} with parameters $(\rho,\beta_1,\beta_2^{(n)},\beta_3)$, where $\beta_2^{(n)}:=\beta_2+n-1$;
				\item[(ii)] We have for $x \in \beta_3 +\mathbb {N}_0$,
				\begin{align*}
					\ell^{(n)}(x) &= \tfrac{1}{(1-\alpha)^{n-1}}\E\left[ \tfrac{(1+\alpha U)^{n-1}}{(1+U)^{x-\beta_3+n-1}}\right] =  \left(\tfrac{\alpha}{1-\alpha}\right)^{n-1} L_U(x-\beta_3,-(n-1),n-1). 
				\end{align*}
			\end{enumerate}
		\end{lemma}
		\begin{proof}
			(i) We proceed by induction. We already know that $\ell=\ell^{(1)}$ satisfies \eqref{hde} with the appropriate parameters. Assume that $\ell^{(n-1)}$ satisfies \eqref{hde} with the parameters $(\rho,\beta_1,\beta_2^{(n-1)}, \beta_3)$. Then, for $x\in \beta_3+\mathbb N_0$, the difference between the left-hand sides of \eqref{hde} at $x+1$ and the left-hand sides of \eqref{hde} evaluated at $x$, multiplied by $\rho$, gives 
			\begin{multline*}
				(x+2+\beta_1+\beta_2^{(n-1)}) \delta_{\rho} \ell^{(n-1)}(x+2)\\
				- 
				\left( (1+\rho)(x+1)+\beta_1\rho+\beta_2^{(n-1)}\right) \delta_{\rho} \ell^{(n-1)}(x+1)+ \rho x\delta_{\rho} \ell^{(n-1)}(x)\\
				+  \delta_{\rho} \ell^{(n-1)}(x+2)-\delta_{\rho} \ell^{(n-1)}(x+1).
			\end{multline*}
            Thus, the above difference equation can be rewritten as
			\begin{multline*}
				(x+2+\beta_1+(\beta_2^{(n-1)}+1)) \ell^{(n)}(x+2)\\
				- 
				\left( (1+\rho)(x+1)+\beta_1\rho+(\beta_2^{(n-1)}+1)\right) \ell^{(n)}(x+1)\\
				+ \rho x \ell^{(n)}(x)=0.
			\end{multline*}
			Since $\beta_2^{(n)} = \beta_2^{(n-1)}+1$, the proof of (i) is complete. 
			
			(ii) We proceed again by induction. The case $n=1$ follows from the definition of $\ell^{(1)}(x)=\ell(x)=L_U(x -\beta_3,0,0)$. Assume that the assertion holds for $n\in\N$, and consider
			\begin{align*}
				\ell^{(n+1)}(x) &= \delta_\rho\,\ell^{(n)}(x)=\ell^{(n)}(x+1) - \rho \ell^{(n)}(x) \\
				&= \tfrac{1}{(1-\alpha)^{n-1}}\E\left[ \tfrac{(1+\alpha U)^{n-1}}{(1+U)^{x-\beta_3+n}}\right] +\tfrac{\alpha}{(1-\alpha)^{n}}\E\left[ \tfrac{(1+\alpha U)^{n-1}}{(1+U)^{x-\beta_3+n-1}}\right]\\
				&=\tfrac{1}{(1-\alpha)^n}\E\left[\tfrac{(1+\alpha U)^{n-1}}{(1+U)^{x-\beta_3+n}}\left(1-\alpha + \alpha(1+U)\right) \right] \\
				&= \tfrac{1}{(1-\alpha)^n}\E\left[\tfrac{(1+\alpha U)^{n}}{(1+U)^{x-\beta_3+n}}\right]= (\tfrac {\alpha}{1-\alpha})^{n} L_U(x-\beta_3,-n, n),
			\end{align*}
			where the last equality follows from the definition of $L_U$. This proves the representation formula in (ii).
		\end{proof}

   The representation of solutions to equations of the form \eqref{hde} can be extracted from \cite[pp.~95--96]{HDE}. The formulation there is given in the complex domain. To avoid discussion of branch choices, in particular for the term \((-1)^{x-\beta_3}\), we give a direct  proof adapted to our setting of $x-\beta_3\in \mathbb N_0$.
		\begin{lemma}\label{lem:HDEsolution}
			Consider the linear equation \eqref{hde}, defined for $x \in \beta_3+\mathbb{N}_0$, with $\beta_1,\beta_2>-1$ and $\beta_3>0$. The linear space of its solutions has the basis $\{\ell_1,\ell_2\}$, where 
			\begin{itemize}
            \item[(i)] if $\rho>1$, 
                \begin{equation}\label{solutions1}
                \ell_1(x)=\int_0^{1} t^{x-1}(1-t)^{\beta_1}(\rho-t)^{\beta_2} dt, \qquad \ell_2(x)=\int_{1}^{\rho} t^{x-1}(t-1)^{\beta_1}(\rho-t)^{\beta_2} dt;
                \end{equation}
				\item[(ii)]if $\rho<0$,
             \begin{equation}\label{solutions2}
              \ell_1(x)=\int_0^{1} t^{x-1}(1-t)^{\beta_1}(t-\rho)^{\beta_2}dt, \qquad  \ell_2(x)= 
             (-1)^{x-\beta_3}\int_{\rho}^{0} (-t)^{x-1}(1-t)^{\beta_1}(t-\rho)^{\beta_2} dt.
             \end{equation}
				
			\end{itemize}
		\end{lemma}
		
\begin{proof}
    In both cases it suffices to prove that $\ell_1$ and $\ell_2$ are linearly independent solutions of \eqref{hde}, since the set of solutions is a two-dimensional linear space.

    Denote by $\mathcal Q(\ell)$ the left-hand side of \eqref{hde} for fixed $x\in\beta_3+\mathbb{N}_0$. Elementary computations give the representation 
\begin{align}
\mathcal Q(\ell_i)=\int_{a_i}^{b_i}\,\tfrac{\mathrm d}{\mathrm dt}\,\psi_i(t)\, dt=\psi_i(b_i)-\psi_i(a_i),\label{11}
\end{align}
where the triplets $(\psi(t),\,a,\,b)$ are of the form
\[
(\psi_i(t),a_i,b_i)=
\begin{cases}
\left(t^x(1-t)^{\beta_1+1}(\rho-t)^{\beta_2+1},\,0,\,1\right),
& \rho>1,\ i=1, \\

\left(-t^x(t-1)^{\beta_1+1}(\rho-t)^{\beta_2+1},\,1,\,\rho\right),
& \rho>1,\ i=2, \\

\left(-t^x(1-t)^{\beta_1+1}(t-\rho)^{\beta_2+1},\,0,\,1\right),
& \rho<0,\ i=1, \\

\left((-1)^{x-\beta_3}(-t)^x(1-t)^{\beta_1+1}(t-\rho)^{\beta_2+1},\,\rho,\,0\right),
& \rho<0,\ i=2.
\end{cases}
\]
Since $x,\beta_1+1,\beta_2+1>0$, the functions $\psi_i$, $i=1,2$, vanish at the corresponding endpoints. 
Thus the functions $\ell_i$, $i=1,2$, defined in \eqref{solutions1} and \eqref{solutions2} are solutions to \eqref{hde}.

It remains to show that \(\ell_1\) and \(\ell_2\) are linearly independent.
If \(\rho>1\) then, as \(x\to\infty\),
\[
\ell_1(x)\sim C_1 x^{-\beta_1-1},
\qquad
\ell_2(x)\sim C_2\rho^x x^{-\beta_2-1},
\]
with some constants \(C_i>0\). Since \(\rho>1\) we see \(\ell_1\) and \(\ell_2\) are not proportional.

If \(\rho<0\), write \(x=\beta_3+n\), \(n\in\mathbb N_0\). Then
\[
\ell_1(x)>0,
\qquad
\ell_2(x)=(-1)^n \tilde\ell_2(x),
\]
where \(\tilde\ell_2(x)>0\). Hence \(\ell_2(\beta_3+n)\) changes sign with \(n\), while \(\ell_1(\beta_3+n)\) is always positive. Thus the two solutions are not proportional.

\end{proof}



		\begin{lemma}
			Assume that $|\lambda|<\min\{a,b\}$. Then $U\sim \mathrm{GB}_{II}(-\lambda,a,b;\alpha)$.
		\end{lemma}
		\begin{proof}
			If $\alpha=1$, the result follows directly from Lemma \ref{lem:hde} (ii): this gives, for $\theta \in \mathbb N_0$, 
			\[
			\E\left[\tfrac{1}{(1+U)^{\theta}}\right] = \tfrac{B(b-\lambda,a+\lambda+\theta)}{B(b-\lambda,a+\lambda)} = \int_0^\infty \tfrac{1}{(1+u)^{\theta}}\tfrac{1}{B(b-\lambda,a+\lambda)}\tfrac{u^{b-\lambda-1}}{(1+u)^{a+b}}du,
			\]
			which implies that $U\sim \mathrm{B}_{II}(b-\lambda,a+\lambda)=\mathrm{GB}_{II}(-\lambda,a,b;1)$.
			
			Now, we consider the case $\alpha\neq 1$. 
			As in Lemma  \ref{lem:hde} (i), let $\rho=\alpha/(\alpha-1)$, $\beta_1=b-\lambda-1$, $\beta_2=\lambda-a$, and $\beta_3=a+\lambda$. By assumption, we have $\beta_1>-1$ and $\beta_3>0$, but $\beta_2<0$. 
            
To apply Lemma \ref{lem:HDEsolution}, whose assumption is \(\beta_2>-1\), we use Lemma \ref{lem:HDEn} to shift the second parameter from \(\beta_2\) to
\(\beta_2^{(n)}=\beta_2+n-1\), and fix \(n\in\mathbb{N}\) so that \(\beta_2^{(n)}>-1\).       
			
			We will consider two cases  $\alpha>1$ and $\alpha\in(0,1)$, separately:
			
			\begin{enumerate}
				\item {\bf Case  $\alpha>1$.} In this case $\rho=\tfrac{\alpha}{\alpha-1}>1$. Applying Lemma \ref{lem:HDEn} (ii) and Lemma \ref{lem:HDEsolution} (i), 
                %
                we obtain 
                for
                $x \in \beta_3+\mathbb{N}_0$, 
				\begin{align*}
					\ell^{(n)}(x)= \delta_1 \ell_1(x) 
					+\delta_2 \ell_2(x)
				\end{align*}
                for some constants $\delta_1$ and $\delta_2$. Note that $\ell_2(x)$  diverges to $\infty$ as $x\to\infty$.  By Lemma \ref{lem:HDEn} (ii), with $x=\beta_3+\theta$, $\theta\in\mathbb N_0$, we have  $$
                \ell^{(n)}(\beta_3+\theta)=\frac {1}{(1-\alpha)^{n-1}}\E\left[ \tfrac{(1+\alpha U)^{n-1}}{(1+U)^{\theta+n-1}}\right] \stackrel{\theta\to\infty}{\longrightarrow}  0,
                $$
      which allows us to conclude that $\delta_2=0$. As a consequence, as a function of $\theta\in\mathbb{N}_0$,
				\begin{align*}
					\E\left[ \tfrac{(1+\alpha U)^{n-1}}{(1+U)^{\theta +n-1}}\right] &\propto  \int_0^{1} t^{\theta +\beta_3-1}(1-t)^{\beta_1}\left(\rho-t\right)^{\beta_2+n-1} dt  \\
					&\propto \int_0^\infty \tfrac{(1+\alpha u)^{n-1}}{(1+u)^{\theta+n-1}} \tfrac{u^{\beta_1}(1+\alpha u )^{\beta_2}}{(1+u)^{\beta_1+\beta_2+\beta_3+1}} du \\
				&	= \int_0^\infty \tfrac{(1+\alpha u)^{n-1}}{(1+u)^{\theta+n-1}} \tfrac{u^{b-\lambda-1}}{(1+\alpha u)^{a-\lambda}(1+u)^{b+\lambda}}du,
                    \end{align*}
				
                with the substitution $t=(1+u)^{-1}$ in the integral. Therefore $U$ has a density proportional to
				\begin{equation}\label{exactlaw}
				\tfrac{u^{b-\lambda-1}}{(1+\alpha u)^{a-\lambda}(1+u)^{b+\lambda}} \mathbf 1_{(0,\infty)}(u),
				\end{equation}
				i.e., $U\sim\mathrm{GB}_{II}(-\lambda,a,b;\alpha)$. 
				
				\item {\bf Case  $\alpha\in(0,1)$.} Here $\rho<0$. By Lemma \ref{lem:HDEn} and Lemma \ref{lem:HDEsolution} (ii), 
                $x \in \beta_3+\mathbb{N}_0$
				\begin{align}\label{lnx}
					 \ell^{(n)}(x) =
					 \delta_1 \ell_1(x) 
					+ \delta_2 \ell_2(x)
		\end{align}
       for some constants $ \delta_1$ and $ \delta_2$.
Setting $x = \beta_3+\theta$, $\theta\in\mathbb N$, and $d_i = (1-\alpha)^{n-1} \delta_i$, for $i=1,2$, using  Lemma \ref{lem:HDEn} (ii), \eqref{lnx} can be rewritten as 
\begin{align}\label{lnx_}
J(\theta):=\E\left[\frac{(1+\alpha U)^{n-1}}{(1+U)^{\theta+n-1}}\right]
= d_1 \ell_1(\beta_3+\theta)+d_2(-1)^{\theta}\tilde\ell_2(\theta),\qquad \theta\in\mathbb N_0,
\end{align}
where
\[
\tilde \ell_2(\theta)=\int_{\rho}^{0} (-t)^{\theta+\beta_3-1}(1-t)^{\beta_1}(t-\rho)^{\beta_2+n-1}\,dt .
\]

Since
\[
\frac{1+\alpha U}{1+U}=\alpha+(1-\alpha)A\qquad \mbox{with }A:=\frac1{1+U}\in(0,1],
\]
we have, writing the law of $A$ as $\mathbb {P}_A$,
\[
J(\theta)=\int y^\theta\,\mu_0(dy),
\qquad
\mu_0(dy):=(\alpha+(1-\alpha)y)^{n-1}\,\P_A(dy).
\]
where \(\mu_0\) is a nonzero finite measure on \([0,1]\).

Likewise,
\[
\ell_1(\theta+\beta_3)=\int t^\theta\,\mu_1(dt),
\qquad
\mu_1(dt):=t^{\beta_3-1}(1-t)^{\beta_1}(t-\rho)^{\beta_2+n-1}\mathbf 1_{(0,1)}(t)\,dt,
\]
and after the change of variables \(r=-t\), 
\[
\tilde\ell_2(\theta)=\int_0^{-\rho} r^{\theta+\beta_3-1}(1+r)^{\beta_1}(-\rho-r)^{\beta_2+n-1}\,dr
=\int r^\theta\,\mu_2(dr),
\]
where
\[
\mu_2(dr):=r^{\beta_3-1}(1+r)^{\beta_1}(-\rho-r)^{\beta_2+n-1}\mathbf 1_{(0,-\rho)}(r)\,dr.
\]
By the assumptions \(\beta_1>-1\), \(\beta_3>0\), \(\beta_2+n-1>-1\), both \(\mu_1\) and \(\mu_2\) are nonzero finite  compactly supported measures in \((0,\infty)\).

As a consequence, \eqref{lnx_} can be rewritten as
\begin{equation}\label{d1d2}
\int y^\theta\,\mu_0(dy)
=
d_1\int y^\theta\,\mu_1(dy)
+
d_2\int (-y)^\theta\,\mu_2(dy),
\qquad \theta \in\mathbb N_0.
\end{equation}
This is enough to claim that $d_2=0$. Indeed,
since \(\mu_j\), \(j=0,1,2\), are finite and compactly supported,  their Laplace transforms can be represented by the moment series
\[
F_j(z):=\int e^{zy}\,\mu_j(dy)=\sum_{\theta=0}^\infty\,\tfrac{z^\theta}{\theta!}\,\int\,y^\theta\,\mu_j(dy),\quad z\in\mathbb R,\qquad j=0,1,2.
\]
In view of \eqref{d1d2}, we get
\[
F_0(z)=d_1F_1(z)+d_2F_2(-z),\qquad z\in\mathbb R.
\]
Letting \(z\to-\infty\), we get
\[
F_0(z)\to0,\qquad F_1(z)\to0,
\]
because \(\mu_0,\mu_1\) are supported in \((0,\infty)\). On the other hand, since \(\mu_2\neq0\), there exists \(\varepsilon>0\) such that \(\mu_2([\varepsilon,-\rho])>0\), and therefore
\[
F_2(-z)=\int e^{-zy}\,\mu_2(dy)
\ge e^{-z\varepsilon}\mu_2([\varepsilon,-\rho]) \stackrel{z\to-\infty}{\longrightarrow}\infty.
\]
This is impossible unless \(d_2=0\), which implies \( \delta_2=0\).

Finally, changing the variable $t=\tfrac{1}{1+u}$ in the integral defining  $\ell_1(\beta_3+\theta)$, we obtain for $\theta\in \mathbb N_0$,
		\begin{align*}
					\E\left[ \tfrac{(1+\alpha U)^{n-1}}{(1+U)^{\theta+n-1}}\right]
					\propto \int_0^\infty \tfrac{(1+\alpha u)^{n-1}}{(1+u)^{\theta+n-1} } \tfrac{u^{b-\lambda-1}}{(1+\alpha u)^{a-\lambda}(1+u)^{b+\lambda}}du.
				\end{align*}
				This proves that $U$ has a density proportional to 
            \eqref{exactlaw},
                as in the previous case.
                
			\end{enumerate}
		\end{proof}

		\subsection{Completion of the proof and matching of parameters}\label{sec:completion}

		At this stage, we have already shown that $U\sim \mathrm{GB}_{II}(-\lambda,a,b;\alpha)$. Thus, \eqref{final} yields $X\sim \mathrm{GB}_{II}(\lambda,a,b;\alpha)$. It remains to identify the laws of $Y$ and $V$. 

		Due to the symmetry between $(M_X,M_U)$ and $(M_Y,M_V)$ in \eqref{sym1} and \eqref{sym2}, we conclude that $$Y\sim \mathrm{GB}_{II}(\widetilde\lambda,\widetilde a,\widetilde b;\beta)\quad \mbox{and}\quad V\sim \mathrm{GB}_{II}(-\widetilde\lambda,\widetilde a,\widetilde b;\beta).$$ Referring to \eqref{Lindep}, in view of \eqref{LX/LU}, we get
		$$
		\tfrac{(a-\lambda)^{(\sigma)}\,(b+\lambda)^{(s+\theta)}}{(a+\lambda)^{(\theta)}\,(b-\lambda)^{(s+\sigma)}}\stackrel{\eqref{LX/LU}}{=}\tfrac{L_X^{(\alpha)}(\theta,\sigma,s)}{L_U^{(\alpha)}(\sigma,\theta,s)}\stackrel{\eqref{Lindep}}{=}\tfrac{L_V^{(\beta)}(\theta,\sigma,s)}{L_Y^{(\beta)}(\sigma,\theta,s)}\stackrel{\eqref{LX/LU}}{=}\tfrac{(\widetilde a+\widetilde \lambda)^{(\sigma)}\,(\widetilde b-\widetilde\lambda)^{(\theta+s)}}{(\widetilde a-\widetilde \lambda)^{(\theta)}\,(\widetilde b+\widetilde \lambda)^{(\sigma+s)}}.$$
		Comparing the first and the last expression in the above sequence of equalities, valid for any $s,\theta,\sigma\in\mathbb N_0$, we conclude that $\widetilde\lambda=-\lambda$, $\widetilde a=a$, $\widetilde b=b$, which ends the proof of Theorem \ref{charac}. 

\subsection*{Acknowledgments}
This research was funded in part by National Science Centre, Poland, \\\mbox{2023/51/B/ST1/01535}. We thank the anonymous referees for their detailed comments, which led to improvements in the paper.

		\bibliographystyle{amsplain}
		\bibliography{BETAII}

	\end{document}